\documentclass[12pt]{article}
\usepackage{a4}
\usepackage{amsmath,amssymb,amsthm, graphicx}
\usepackage[dvips]{epsfig}
\usepackage{nomencl}

\newtheorem{theorem}{Theorem}[section]
\newtheorem{proposition}[theorem]{Proposition}
\newtheorem{claim}[theorem]{Claim}
\newtheorem{lemma}[theorem]{Lemma}

\makenomenclature

\newcommand{\en}{\mathbb N}
\newcommand{\er}{\mathbb R}
\newcommand{\norm}{\|}

\DeclareMathOperator{\glued}{glued}
\DeclareMathOperator{\merge}{merge}
\DeclareMathOperator{\conv}{conv}
\DeclareMathOperator{\interior}{int}
\DeclareMathOperator{\dist}{dist}
\DeclareMathOperator{\STR}{STR}
\DeclareMathOperator{\CON}{CON}
\DeclareMathOperator{\SEG}{SEG}
\DeclareMathOperator{\susp}{susp}

\newcommand{\set}[1]{\left\{ #1\right\}}
\newcommand{\setcond}[2]{\set{ #1  \left| \ #2 \right.}}
\newcommand{\scf}[1]{\mathsf{#1}}
\newcommand{\lambdamergeone}{\lambda_{\merge, 1}}
\newcommand{\lambdamergetwo}{\lambda_{\merge, 2}}
\newcommand{\iotamerge}{\iota_{\merge}}
\newcommand{\sigmab}{\sigma_B}
\newcommand{\sigmag}{\sigma_G}
\renewcommand{\AA}{\scf A}
\newcommand{\BB}{\scf B}
\newcommand{\CC}{\scf C}
\newcommand{\CCg}{\scf C_{\glued}}
\newcommand{\DD}{\scf D}
\newcommand{\FF}{\scf F}
\newcommand{\GG}{\scf G}
\newcommand{\Ff}{\mathcal F}
\newcommand{\HH}{\scf H}
\newcommand{\II}{\scf I}
\newcommand{\JJ}{\scf J}
\newcommand{\KK}{\scf K}
\newcommand{\LL}{\scf L}
\newcommand{\LLp}{\overline{\scf L}}
\newcommand{\MM}{\scf M}
\newcommand{\NN}{\scf N}
\newcommand{\OO}{\scf O}
\newcommand{\RR}{\scf R}
\renewcommand{\SS}{\scf S}
\newcommand{\TT}{\scf T}
\newcommand{\VV}{\scf V}
\newcommand{\TWO}{\scf 2}
\renewcommand{\aa}{{\mathbf a}}
\newcommand{\bb}{{\mathbf b}}
\newcommand{\ee}{{\mathbf e}}
\renewcommand{\ss}{{\mathbf s}}
\newcommand{\uu}{{\mathbf u}}
\newcommand{\vv}{{\mathbf v}}
\newcommand{\ww}{{\mathbf w}}
\newcommand{\xx}{{\mathbf x}}

\newcommand{\collapseto}{\twoheadrightarrow}
\newcommand{\St}{\mathop{\hbox{St}}}

\newcommand{\heading}[1]{\medskip\par\noindent{\bf #1}}

\makeatletter
\newcommand{\ProofEndBox}{{\ifhmode\unskip\nobreak\hfil\penalty50 \else
          \leavevmode\fi\quad\vadjust{}\nobreak\hfill\qedsymbol
                      \finalhyphendemerits=0 \par}}%
\makeatother
\newcommand{\proofend}{\ProofEndBox\smallskip}

\title{$d$-collapsibility is NP-complete for $d \geq 4$}
\author{
Martin Tancer\thanks{
Department of Applied Mathematics and Institute for Theoretical Computer Science (supported by project 1M0545
of The Ministry of Education of the Czech Republic), Faculty of Mathematics
and Physics, Charles University, Malostransk\'e n\'am.~25, 118~00 Prague,
Czech Republic. Partially supported by project GAUK 49209.
E-mail: {\tt tancer@kam.mff.cuni.cz}
}}
\begin{document}

\maketitle
\begin{abstract}
A simplicial complex is \emph{$d$-collapsible} if it can be reduced to an empty
complex by repeatedly removing (collapsing) a face of dimension at most $d-1$
that is contained in a unique maximal face. We prove that the algorithmic
question whether a given simplicial complex is $d$-collapsible is NP-complete
for $d \geq 4$ and polynomial time solvable for $d \leq 2$.

As an intermediate step, we prove that $d$-collapsibility can be recognized by
the greedy algorithm for $d \leq 2$, but the greedy algorithm does not work for
$d \geq 3$.

A simplicial complex is \emph{$d$-representable} if it is the nerve of a
collection of convex sets in $\er^d$. The main motivation for studying
$d$-collapsible complexes is that every $d$-representable complex is
$d$-collapsible. We also observe that known results imply that 
$d$-representability is NP-hard to decide for $d \geq
2$.
\end{abstract}

\section{Introduction}

Our task is to determine the computational complexity of recognition of
\emph{$d$-collapsible} simplicial complexes. These complexes were introduced
by Wegner~\cite{wegner75} and studying them is motivated by
Helly-type theorems, which we will discuss later. All the simplicial
complexes\footnote{We assume that the reader is familiar with simplicial
complexes; introductory chapters of books
like~\cite{matousek03,hatcher01,munkres84} should provide a sufficient
background. Unless stated otherwise, we work with abstract simplicial
complexes, i.e., set systems $\KK$ such that if $A \in \KK$ and $B \subseteq A$
then $B \in \KK$.} throughout the article are assumed to be finite.

\subsection{Main results}

\heading{\boldmath $d$-collapsible complexes.}
Informally, a simplicial complex is $d$-collapsible if it can be vanished by
removing faces of dimension at most $d-1$ which are contained in unique maximal
faces. A more detailed motivation for this definition is explained after
introducing $d$-representable complexes. Next we introduce some notation and state a precise definition.

A face $\sigma$ of a simplicial complex $\KK$ is \emph{collapsible} if there is
a unique maximal face of $\KK$ containing $\sigma$ (by ``maximal face'' we
always mean ``inclusionwise-maximal face''). Unless stated otherwise,
we denote this maximal face by 
$\tau(\sigma)$. 
(We allow $\tau(\sigma)
 = \sigma$.)
Moreover, if $\dim \sigma \leq d-1$, then $\sigma$ is \emph{$d$-collapsible}.
By $[\sigma, \tau(\sigma)]$ we denote the set
$$
\setcond{\eta \in \KK}{\sigma \subseteq \eta \subseteq \tau(\sigma)}
$$
of faces of $\KK$ that contain $\sigma$.
\nomenclature{$\tau(\sigma)$}{the unique maximal face containing $\sigma$
(assuming that such a face exists)\nomrefpage}
\nomenclature{$[\sigma,\tau(\sigma)]$}{the set of faces $\eta$ such that
$\sigma \subseteq \eta \subseteq \tau(\sigma)$ (assuming that $\tau(\sigma)$
exists)\nomrefpage}

We assume that $\sigma$ is $d$-collapsible and we say that the complex $\KK' =
\KK \setminus [\sigma, \tau(\sigma)]$ arises from $\KK$ by an \emph{elementary
$d$-collapse}. In symbols, 
$$
\KK \rightarrow \KK'.
$$ 
If we want to stress $\sigma$ we write
$$
\KK' = \KK_{\sigma}.
$$
\nomenclature{$\KK \rightarrow \KK'$}{a complex $\KK'$ arises from $\KK$ by an
elementary $d$-collapse\nomrefpage}
\nomenclature{$\KK_\sigma$}{a complex obtained by collapsing a $d$-collapsible
faces $\sigma$ in complex $\KK$\nomrefpage}

A complex $\KK$ \emph{$d$-collapses} to a complex $\LL$, in symbols $\KK
\collapseto \LL$, if there is a sequence of elementary $d$-collapses
$$
\KK \rightarrow \KK_2 \rightarrow \KK_3 \rightarrow \cdots \rightarrow
\LL.
$$
This sequence is called a \emph{$d$-collapsing} (of $\KK$ to $\LL$).
Finally, a complex $\KK$ is \emph{$d$-collapsible} if $\KK \collapseto \emptyset$.
An example of 2-collapsible complex is in Figure~\ref{Fig2coll}.
\nomenclature{$\KK \collapseto \LL$}{a complex $\KK$ $d$-collapses to a complex
$\LL$\nomrefpage}

\heading{\boldmath The computational complexity of $d$-collapsibility.}
How hard is it to decide whether a given simplicial complex is
$d$-collapsible? We consider the computational complexity of this question (the size of an input
is the number of faces of the complex in the question), regarding $d$ as a fixed integer; we refer to it as
$d$-COLLAPSIBILITY.

According to Lekkerkerker and Boland~\cite{lekkerkerker62} (see also~\cite{wegner75}), 
$1$-collapsible complexes are
exactly clique complexes over \emph{chordal graphs}. (A graph is chordal if
it does not contain an induced cycle of size 4 or more.) 
1-COLLAPSIBILITY is therefore polynomial time solvable. (Polynomiality of
1-COLLAPSIBILITY
also follows from Theorem~\ref{ThmOrder}(i).)

The main result of this paper is the following:

\begin{theorem}
\label{ThmMain}
\begin{itemize}
\item[\emph{(i)}] $2$-COLLAPSIBILITY is polynomial time solvable.
\item[\emph{(ii)}] $d$-COLLAPSIBILITY is NP-complete for $d \geq 4$.
\end{itemize}
\end{theorem} 

\begin{figure}
\begin{center}
\epsfbox{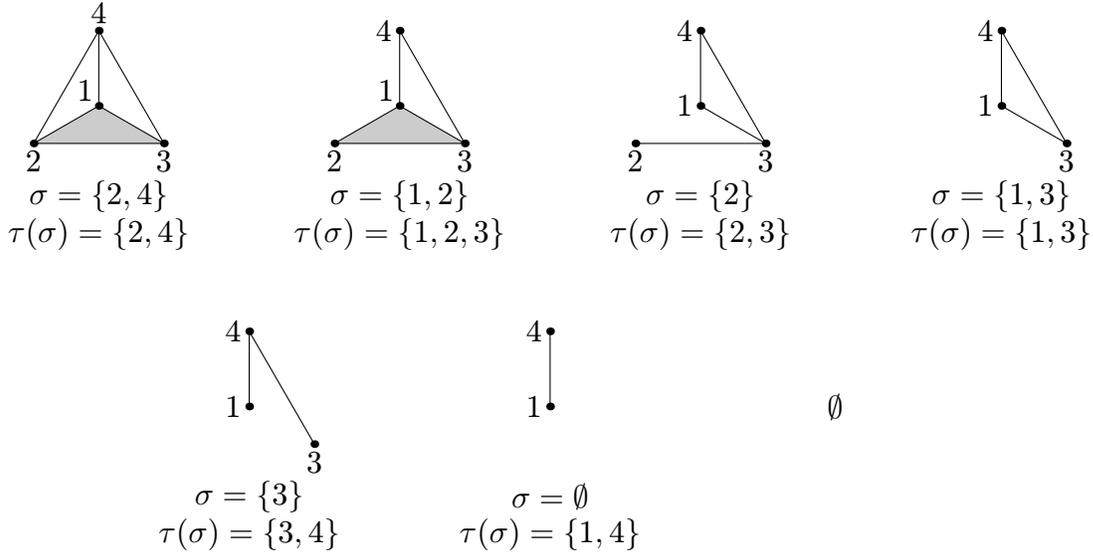}
\caption{An example of $2$-collapsing.}
\label{Fig2coll}
\end{center}
\end{figure}

Suppose that $d$ is fixed. A \emph{good face} is a $d$-collapsible face 
of $\KK$ such that $\KK_{\sigma}$ is $d$-collapsible; a \emph{bad face} is
a $d$-collapsible face of $\KK$ such that $\KK_{\sigma}$ is not $d$-collapsible. 

Now suppose that $\KK$ is a $d$-collapsible complex. It is not immediately clear
whether we can choose elementary $d$-collapses greedily in any order to
$d$-collapse $\KK$, or whether there is a ``bad sequence'' of $d$-collapses such that the resulting
complex is no longer $d$-collapsible. Therefore, we consider the following
question: For which $d$ there is a $d$-collapsible complex $\KK$ such that it
contains a bad face? The answer is:

\begin{theorem}
\label{ThmOrder}
\begin{itemize}
\item[\emph{(i)}] Let $d \leq 2$. Then every $d$-collapsible face of a
$d$-collapsible complex is good.
\item[\emph{(ii)}] Let $d \geq 3$. Then there exists a $d$-collapsible complex
containing a bad $d$-collapsible face.
\end{itemize} 
\end{theorem}

Theorem~\ref{ThmMain}(i) is a straightforward consequence of
Theorem~\ref{ThmOrder}(i). Indeed, if we want to test whether a given complex
is 2-collapsible, it is sufficient to greedily collapse $d$-collapsible faces.
Theorem~\ref{ThmOrder}(i) implies that we finish with an empty complex if
and only if the original complex is 2-collapsible.

Our construction for Theorem~\ref{ThmOrder}(ii) is an
intermediate step to proving Theorem~\ref{ThmMain}(ii).

\subsection{Motivation and background}

\heading{\boldmath $d$-representable complexes.}
Helly's theorem~\cite{helly23} asserts that if $C_1, C_2, \dots, C_n$
are convex sets in $\er^d$, $n \geq d + 1$, and every $d + 1$ of them have a
common point, then $C_1 \cap C_2 \cap \dots \cap C_n \neq \emptyset$. This
theorem (and several other theorems in discrete geometry) deals with
\emph{intersection patterns} of convex sets in $\er^d$. It can be restated
using the notion of \emph{$d$-representable} complexes, which ``record'' the
intersection patterns.

The \emph{nerve} of a family $\mathcal{S} = \set{S_1, S_2, \dots, S_n}$ of sets
is the simplicial complex with vertex set $[n] = \set{1, 2, \dots, n}$ and with
the set
$\sigma \subseteq [n]$ forming a face if $\bigcap_{i \in \sigma} S_i \neq
\emptyset$. A simplicial complex is $d$-representable if it is isomorphic to
the nerve of a family of convex sets in $\er^d$.

\nomenclature{$[n]$}{the set $\set{1,2,\dots,n}$\nomrefpage}

In this language, Helly theorem states that if a $d$-representable complex
(with the vertex set $V$) contains all faces of dimension at most $d$, then
it is already a \emph{full simplex} $\TWO^V$. Beside Helly's theorem we also
mention several other known results that can be formulated using
$d$-representability.
They include the fractional Helly theorem of Katchalski and
Liu~\cite{katchalski-liu79}, the
colorful Helly theorem of Lov\' asz (\cite{lovasz74}; see
also~\cite{barany82}), the $(p,
q)$-theorem of Alon and Kleitman~\cite{alon-kleitman92}, and the Helly type
result of
Amenta~\cite{amenta96}.
\nomenclature{$\TWO^V$}{the full simplex on a vertex set $V$\nomrefpage}

\heading{\boldmath $d$-Leray complexes.}
Another related notion is a \emph{$d$-Leray} simplicial complex, where $\KK$ is
$d$-Leray if every \emph{induced subcomplex} of $\KK$ (i.e. a subcomplex of
the form $\KK[X] = \setcond{\sigma \cap X}{\sigma \in \KK}$ for some subset $X$
of the vertex set of $\KK$) has zero homology (over $\mathbb Q$) in dimensions
$d$ and larger. We will mention $d$-Leray complexes only briefly, thus the
article should be accessible also for the reader not familiar with homology.

\heading{Relations among the preceding notions.} 
Wegner~\cite{wegner75} proved
that $d$-representable complexes are $d$-collapsible and also that 
$d$-collapsible complexes are $d$-Leray. 

Regarding the first inclusion, suppose that we are given convex sets in $\er^d$ 
representing a $d$-representable complex. Sliding a generic hyperplane (say
from infinity to minus infinity) and cutting off the pieces on the positive
side of the hyperplane yields a $d$-collapsing of the complex. (Several
properties have to be checked, of course.) This is the main motivation for the
definition of $d$-collapsibility.

The second inclusion is more-less trivial (for a reader familiar with homology)
since $d$-collapsing does not affect homology of dimension $d$ and larger.

Many results on $d$-representable complexes can be generalized in terms of
$d$-collapsible complexes, the results mentioned here even for $d$-Leray
complexes.

For example, a topological generalization of Helly's theorem follows from
Helly's own work~\cite{helly30}, a generalization of the fractional Helly
theorem and
$(p,q)$-theorem was done in~\cite{alon-kalai-matousek-meshulam02}, and a
generalization of the colorful Helly
theorem and Amenta's theorem was proved by Kalai and
Meshulam~\cite{kalai-meshulam05},~\cite{kalai-meshulam08}.

Dimensional gaps between collapsibility and
representability were studied by Matou\v{s}ek and the
author~\cite{matousek-tancer09, tancer10}; an interesting variation on
$d$-collapsibility was
used by Matou\v{s}ek in order to show that it is not easy to remove degeneracy
in LP-type problems~\cite{matousek09}.


\heading{Related complexity results.}
Similarly as $d$-COLLAPSIBILITY, we can also consider the computational
complexity of $d$-REPRESENTABILITY and $d$-LE\-RAY COMPLEX.

By a modification of a result of Kratochv\'\i l and Matou\v{s}ek on string
graphs (\cite{kratochvil-matousek89}; see also \cite{kratochvil91}), one has that
$2$-REPRESENTABILITY is NP-hard. Moreover, this result also implies that
$d$-REPRESENTABILITY is NP-hard for $d \geq 2$. Details are given in
Section~\ref{AppRepresentability}. It is not known to the author whether
$d$-REPRESENTABILITY belongs to NP.

Finally, $d$-LERAY COMPLEX is polynomial time solvable, since an equivalent characterization
of $d$-Leray complexes is that it is sufficient to test whether the homology
(of dimension greater or equal to $d$) of \emph{links}\footnote{A link of a
face $\sigma$ in a complex $\KK$ is the complex $\setcond{\eta \in \KK}{\eta
\cup \sigma \in \KK, \eta \cap \sigma = \emptyset}$.} of faces of the
complex in the question vanishes. These tests can be performed in a polynomial
time; see~\cite{munkres84} (note that the $k$-th homology of a complex of dimension less than $k$ is
always zero; note also that the homology is over $\mathbb Q$, which simplifies
the situation---computing homology for this case is indeed only a linear algebra).

Among the above mentioned notions, $d$-REPRESENTABILITY is of the biggest
interest since it straightly affects intersection patterns of convex sets.
However, NP-hardness of this problem raises the question, whether it can be
replaced with $d$-COLLAPSIBILITY or $d$-LERAY COMPLEX. As we have already
mentioned, $d$-LERAY COMPLEX is polynomial time solvable thus one could be
satisfied with replacing $d$-REPRESENTABILITY with $d$-LERAY COMPLEX. However,
$d$-COLLAPSIBILITY is closer to $d$-REPRESENTABILITY. 

One of the important differences regards Wegner's conjecture. An open set in
$\er^d$ homeomorphic to an open ball is a \emph{$d$-cell}. A \emph{good
cover} in $\er^d$ is a collection of $d$-cells such that an intersection of any
subcollection is again a $d$-cell or empty. Wegner conjectured that nerve of a finite good cover in $\er^d$ is $d$-collapsible. 
A recent result of the author disproves this conjecture~\cite{tancer10prep}. However,
the nerve of a finite good cover in $\er^d$ is always $d$-Leray due to the
nerve theorem; see, e.g,~\cite{bjorner95,borsuk48}.
We get that the notion of $d$-Leray complexes cannot distinguish the nerves of
collections of convex sets and good covers; however, $d$-representability is
stronger in this respect.
That is also why we want to clarify the complexity status of $d$-COLLAPSIBILITY. 


\heading{A particular example of computational interest.} A collection of
convex sets in $\er^d$ has a \emph{$(p,q)$-property} with $p \geq q \geq d+1$
if among every $p$ sets of the collection there is a subcollection of $q$ sets 
with a nonempty intersection. The above mentioned $(p,q)$-theorem of Alon and
Kleitman states that for all integers $p, q, d$ with $p \geq q \geq d+1$ there
is an integer $c$ such that for every finite collection of convex sets in
$\er^d$ with $(p,q)$-property there are $c$ points in $\er^d$ such that every
convex set of the collection contains at least one of the selected points. Let
$c' = c'(p,q,d)$ be the minimum possible value of $c$ for which the conclusion
of the $(p,q)$-theorem holds. A significant effort was devoted to estimating
$c'$. The first unsolved case regards
estimating $c'(4,3,2)$. The best bounds\footnote{Known to the author.} are due
to Kleitman, Gy\'arf\'as and T\'oth~\cite{kleitman-gyarfas-toth01}: $3 \leq c'(4,3,2) \leq 13$. It
seems that the actual value of $c'(4,3,2)$ is rather closer to the lower bound
in this case, and thus it would be interesting to improve the lower bound even
by one.\footnote{Kleitman, Gy\'arf\'as and T\'oth offer \$30 for such an improvement.}

 Here $2$-collapsibility could come into the play. When looking for a small
example one could try to generate all $2$-collapsible complexes and check the
other properties.

\heading{Collapsibility in Whitehead's sense.}
Beside $d$-collapsibility, collapsibility in Whitehead's sense is much
better known (called simply \emph{collapsibility}). In the case of collapsibility,
we allow only to collapse a face $\sigma$ that is a proper subface of the
unique maximal face containing $\sigma$. On the other hand, there is no
restriction on dimension of $\sigma$.

Let us mention that one of the important differences between $d$-collapsibility
and collapsibility is that every finite simplicial complex is $d$-collapsible
for $d$ large enough; on the other hand not an every finite simplicial complex is
collapsible.

Malgouyres and Franc\'es~\cite{malgouyres-frances08} proved that it is NP-complete to decide, whether a
given 3-dimensional complex collapses to a given 1-dimensional complex.
However, their construction does not apply to $d$-collapsibility. A key
ingredient of their construction is that collapsibility distinguishes a Bing's
house with thin walls and a Bing's house with a thick wall. However, they are
not distinguishable from the point of view of $d$-collapsibility. They are both
3-collapsible, but none of them is 2-collapsible.

\heading{Technical issues.}
Throughout this paper we will use several technical lemmas about
$d$-collapsibility. Since I think that
the main ideas of the paper can be followed even without these lemmas I decided
to put them separately to Section~\ref{AppTechnical}. The reader is encouraged to skip
them for the first reading and look at them later for full details.

The paper contains many symbols.
For the reader's convenience we add a list of often used symbols.
It is situated at the end of the
paper---just above the bibliography. 

\section{2-collapsibility}

Here we prove Theorem~\ref{ThmOrder}(i).

The case $d = 1$ follows from the fact that $d$-collapsible complexes coincide
with $d$-Leray ones (\cite{lekkerkerker62,wegner75}). Indeed, let $\KK$ be a 1-collapsible complex and
let $\sigma$ be its 1-collapsible face. We have that $\KK$ is 1-Leray, which
implies that $\KK_{\sigma}$ is 1-Leray (1-collapsing does not affect homology
of dimensions 1 and more). This implies that $\KK_{\sigma}$ is 1-collapsible,
i.e., $\sigma$ is good. In fact, the case $d = 1$ can be also solved by a similar
(simpler) discussion as the following case $d=2$.  
 

It remains to consider the case $d=2$.
Suppose that $\KK$ is a $2$-collapsible complex which, for contradiction,
contains a bad $2$-collapsible face $\sigmab \in \KK$. On the other hand, it
also contains a good face $\sigmag$ since it is $2$-collapsible. Moreover,
we can, without loss of generality, suppose that $\KK$ is the smallest complex (according to the number of faces) with these properties.


\begin{claim}
\label{Clagood}
Let $\sigma$ be a good face of \emph{$\KK$} and let $\sigma'$ be a $2$-collapsible face of \emph{$\KK_{\sigma}$}. Then $\sigma'$ is a good face of \emph{$\KK_{\sigma}$}.
\end{claim}

\begin{proof}
The complex $\KK_{\sigma}$ is $2$-collapsible since $\sigma$ is a good face
of $\KK$. If $\sigma'$ were a bad face of $\KK_{\sigma}$, then $\KK_{\sigma}$
would be a smaller counterexample to Theorem~\ref{ThmOrder}(i) contradicting the choice of $\KK$.
\end{proof}

Recall that $\tau(\sigma)$ denotes the unique maximal superface of a
collapsible face $\sigma$. Two collapsible faces $\sigma$ and $\sigma'$ are
\emph{independent} if $\tau(\sigma) \neq \tau(\sigma')$; otherwise, they are
\emph{dependent}.
 The symbol $\St(\sigma, \KK)$ denotes the (open) \emph{star} of a face
$\sigma$ in $\KK$, which consists of all superfaces of $\sigma$ in $\KK$
(including $\sigma$). We remark that $\St(\sigma, \KK) = [\sigma,
\tau(\sigma)]$ in case that $\sigma$ is collapsible.
\nomenclature{$\St(\sigma, \KK)$}{an open star of a face $\sigma$ in
$\KK$\nomrefpage}

\begin{claim}
\label{Claorder}
Let \emph{$\sigma, \sigma' \in \KK$} be independent $2$-collapsible faces. Then $\sigma$
is a $2$-collapsible face of \emph{$\KK_{\sigma'}$}, $\sigma'$ is a $2$-collapsible
face of \emph{$\KK_{\sigma}$}, and \emph{$(\KK_{\sigma})_{\sigma'} =
(\KK_{\sigma'})_{\sigma}$}.
\end{claim}

\begin{proof}
Since $\tau(\sigma) \neq \tau(\sigma')$, we have $\sigma \not\subseteq
\tau(\sigma')$. Thus, $\St (\sigma, \KK) = \St (\sigma, \KK_{\sigma'})$,
implying that $\tau(\sigma)$ is
also a unique maximal face containing $\sigma$ when considered in
$\KK_{\sigma'}$.  It means that $\sigma$ is a collapsible face of
$\KK_{\sigma'}$. Symmetrically, $\sigma'$ is a collapsible face of
$\KK_{\sigma}$. Finally,
$$ (\KK_{\sigma})_{\sigma'} = (\KK_{\sigma'})_{\sigma}
= \KK \setminus \setcond{\eta \in \KK}{\sigma \subseteq \eta \hbox{ or }
\sigma' \subseteq \eta}.$$
\end{proof}

\begin{claim}
\label{Cladependent}
Any two $2$-collapsible faces of $\KK$ are dependent.
\end{claim}
\begin{proof}
For contradiction, let $\sigma$, $\sigma'$ be two independent $2$-collapsible
faces
in $\KK$. First, suppose that one of them is good, say $\sigma$, and the
second one, i.e., $\sigma'$, is bad. The face $\sigma'$ is a
collapsible face of $\KK_{\sigma}$ by Claim~\ref{Claorder}. Thus,
$(\KK_{\sigma})_{\sigma'}$ is $2$-collapsible by Claim~\ref{Clagood}.
But $(\KK_{\sigma})_{\sigma'} = (\KK_{\sigma'})_{\sigma}$ by
Claim~\ref{Claorder}, which contradicts the assumption that $\sigma'$ is a bad face.

Now suppose that $\sigma$ and $\sigma'$ are good faces. Then at least one of
them is independent of $\sigmab$, which yields the contradiction as in the
previous case. Similarly, if both of $\sigma$ and $\sigma'$ are bad faces, then
at least one of them is independent of $\sigmag$.
\end{proof}
\nomenclature{$\sigmab$}{a bad face\nomrefpage}

Due to Claim~\ref{Cladependent} there exists a universal $\tau \in \KK$ such
that $\tau = \tau(\sigma)$ for every $2$-collapsible $\sigma \in \KK$. Let us
remark that $\KK \neq \TWO^{\tau}$ since $\sigmab$ is a bad face.

The following claim represents a key difference among $2$-collapsibility and
$d$-collapsibility for $d \geq 3$. It wouldn't be valid in case of
$d$-collapsibility.

\begin{claim}
\label{Clagoodbad}
Let $\sigma$ be a good face and let $\sigma'$ be a bad face. Then $\sigma \cap
\sigma' = \emptyset$.
\end{claim}

\begin{proof}
It is easy to prove the claim in the case that either $\sigma$ or $\sigma'$ is
a $0$-face. Let us therefore consider the case that both $\sigma$ and $\sigma'$ are
$1$-faces. For contradiction suppose that $\sigma \cap \sigma' \neq \emptyset$,
i.e., $\sigma = \set{u, v}$, $\sigma' = \set{v, w}$ for some mutually different
$u, v, w \in \tau$. Then $\tau \setminus \set{u}$ is a unique maximal face in
$\KK_{\sigma}$ that contains $\sigma'$, so $(\KK_{\sigma})_{\sigma'}$ exists.
Similarly, $(\KK_{\sigma'})_{\sigma}$ exists and the same argument as in the
proof of Claim~\ref{Claorder} yields $(\KK_{\sigma})_{\sigma'} =
(\KK_{\sigma'})_{\sigma}$. Similarly as in the proof of
Claim~\ref{Cladependent}, $(\KK_{\sigma})_{\sigma'}$ is $2$-collapsible (due to
Claim~\ref{Clagood}), but it contradicts the fact that $\sigma'$ is a bad face.
\end{proof}

The complex $\KK$ is $2$-collapsible. Let
$
\KK = \KK_1 \rightarrow \KK_2 \rightarrow \cdots \rightarrow \KK_m =
\emptyset
$
be a 2-collapsing of $\KK$, where $\KK_{i+1} = \KK_i
\setminus [\sigma_i, \tau_i]$. Clearly, $\tau_1 = \tau$.
Let $k$ be the minimal integer such that $\tau_k \not \subseteq \tau$. Such $k$
exists since $\KK \neq \TWO^{\tau}$. Moreover, we can assume that all the faces
$\sigma_1, \dots, \sigma_{k}$ are edges. This assumption is possible since
collapsing a vertex can be substituted by collapsing the edges connected to the
vertex and then removing the isolated vertex at the very end of the process.
See Lemma~\ref{LemBigFacesFirst} for details.

\begin{claim}
\label{Clatauk}
The face $\sigma_k$ is a subset of $\tau$, and it is not a $2$-collapsible face
of
$\KK$.
\end{claim}

\begin{proof}
Suppose for contradiction that $\sigma_k \not \subseteq \tau$. Then $\St (
\sigma_k, \KK) = \St (\sigma_k, \KK_i)$ since only subsets of $\tau$ are
removed from
$\KK$ during the first $i$ 2-collapses. It implies that $\sigma_k$ is a
$2$-collapsible face of $\KK$ contradicting the definition of $\tau$.

It is not a $2$-collapsible face of $\KK$ since it is contained in $\tau$ and $\tau_k
\not \subseteq \tau$.
\end{proof}

\begin{figure}
\begin{center}
\epsfbox{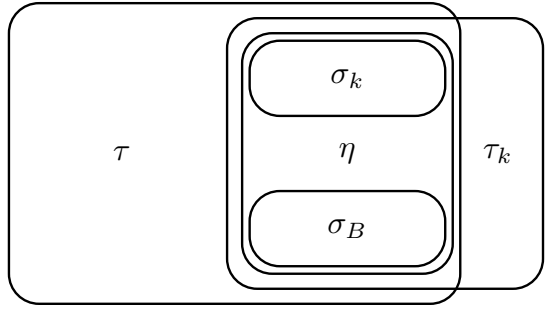}
\caption{The simplices $\tau$, $\tau_k$ and $\eta$.}
\label{Figtau}
\end{center}
\end{figure}

\begin{claim}
\label{Clagoodfaces}
The faces $\sigma_1, \sigma_2, \dots, \sigma_{k-1}$
are good faces of $\KK$. 
\end{claim}
\begin{proof}
First we observe that each $\sigma_i$ is $2$-collapsible face of $\KK$ for $0
\leq i \leq k-1$. If $\sigma_i$ was not $2$-collapsible then there is a face
$\vartheta \in \KK$ containing $\sigma_i$ such that $\vartheta \not \subseteq
\tau$. Then $\vartheta \in \KK_i$ due to minimality of $k$. Consequently 
$\tau_i$ cannot be the unique maximal face of $\KK_i$ containing $\sigma_i$ since
$\vartheta$ contains $\sigma_i$ as well.

In order to show that the faces are good we proceed by induction. The face $\sigma_1$ is a good face of $\KK$ since there is a $d$-collapsing of $\KK$ starting with $\sigma_1$.

Now we assume that $\sigma_1, \dots, \sigma_{i-1}$ are good faces of $\KK$ for
$i \leq k-1$. If there is an index $j < i$ such that $\sigma_j \cap \sigma_i
\neq \emptyset$ then $\sigma_i$ is good by Claim~\ref{Clagoodbad}. 
If this is not the case then we set $\sigma_1 = \{x, y\}$. The faces $\sigma_i
\cup \{x\}$ and $\sigma_i \cup \{y\}$ belong to $\KK_i$; however, $\sigma_1
\cup \sigma_i$ does not belong to $\KK_i$ since $\sigma_1$ was collapsed. Thus
$\sigma_i$ does not belong to a unique maximal face.
\end{proof}

Let $\eta = \sigma_k \cup \sigma_B$. See
Figure~\ref{Figtau}. Claim~\ref{Clatauk} implies that $\eta \subseteq \tau$. By
Claim~\ref{Clagoodbad}
(and the fact that $\sigma_k$ is not a good face---a consequence of Claim~\ref{Clatauk}) the face $\eta$ does not contain a
good face. Thus, $\eta \in \KK_k$ by Claim~\ref{Clagoodfaces}. In particular $\eta \subseteq
\tau_k$ since $\tau_k$ is a unique maximal face of $\KK_k$ containing
$\sigma_k$, hence $\sigma_B \subseteq \tau_k$. On the other hand, $\tau$ is a
unique maximal face of $\KK \supseteq \KK_k$ containing $\sigma_B$ since
$\sigma_B$ is a $2$-collapsible face, which implies $\tau_k \subseteq \tau$. It is a contradiction that $\tau_k \not
\subseteq \tau$.
\proofend

\section{$d$-collapsible complex with a bad $d$-collapse}
\label{SecBad}

In this section we prove Theorem~\ref{ThmOrder}(ii). 

We start with describing the intuition behind the construction. Given a full
complex $\KK = \TWO^S$ (the cardinality of $S$ is $2d$), any $(d-1)$-face is
$d$-collapsible. However, once we collapse one of them, say $\sigma_B$, the
rest $(d-1)$-faces will be divided into two sets, those which are collapsible in
$\KK_{\sigma_B}$ (namely, $\Sigma$), and those which are not (namely, $\bar
\Sigma$). For example, when $d=2$, given a tetrahedron, after collapsing one
edge, among the rest five edges, four are collapsible and one is not. The idea
of the construction is to attach a suitable complex $\CC$ to $\KK$ in such a
way that 
\begin{itemize}
\item the faces of $\Sigma$ are properly contained in faces of $\CC$ (and thus
they cannot be collapsed until $\CC$ is collapsed);
\item there is a sequence of $d$-collapses of some of the faces of $\bar\Sigma$
such that $\CC$ can be subsequently $d$-collapsed.
\end{itemize}
In summary the resulting complex is $d$-collapsible because of the second
requirement. However, if we start with $\sigma_B$, we get stuck because of the
first requirement.

\subsection{Bad complex}
Now, for $d \geq 3$, we construct a \emph{bad} complex $\BB$, which is
$d$-collapsible but it contains a bad face. A certain important but technical
step of the construction is still left out. This is to give the more detailed
intuition to the reader. Details of that step are given in the subsequent
subsections.

\nomenclature{$\BB$}{the bad complex\nomrefpage}

\heading{The complex $\CCg$.}

Suppose that $\sigma$, $\gamma_1,
\dots, \gamma_t$ are already known $(d-1)$-dimensional faces of a given complex
$\LL$.
These faces are assumed to be distinct, but not necessarily disjoint. We
start with the complex $\KK = \TWO^\sigma \cup \TWO^{\gamma_1} \cup \cdots
\cup \TWO^{\gamma_t}$. We attach a certain complex $\CC$ to $\LL'$. The
resulting complex is denoted by $\CCg(\sigma; \gamma_1, \dots, \gamma_t)$. Here
we leave out the details; however, the properties of $\CCg$ are described in
the forthcoming lemma (we postpone the proof of this lemma).

\begin{lemma}
\label{LemCGlued}
Let $\LL$, $\LL'$, and $\CCg = \CCg
(\sigma; \gamma_1, \dots, \gamma_t)$
be the complexes from
the previous paragraph. Then we have:

\begin{itemize}
\item[\emph{(i)}] 
If $\sigma$ is a maximal face of $\LL$, then $\LL \cup \CCg
\collapseto \LL \setminus \set{\sigma}$. 

\item[\emph{(ii)}] The only $d$-collapsible face of $\CCg$ is the face $\sigma$.

\item[\emph{(iii)}] Suppose that $d$ is a constant. Then the number of faces of
$\CCg$ is $O(t)$.
\end{itemize}
\end{lemma}

Let $S = \set{p, q_1, \dots, q_{d-1}, r_1, \dots, r_d}$ be a $2d$-element set.
Consider the full simplex $\TWO^S$. We name its $(d-1)$-faces:\\
\begin{tabular}{rcll}
$\iota$ & $=$ & $\set{p, q_1, \dots, q_{d-1}}$ & is an \emph{initial} face,  \cr
$\lambda_i$ & $=$ &  $\set{q_1, \dots, q_{d-1}, r_i}$ & are
\emph{liberation} faces for $i \in [d]$, \cr
$\sigmab $ & $=$ & $\set{r_1, \dots, r_d}$, & we will show that $\sigmab$ is a
bad  face.  \cr 
\end{tabular}\\
The remaining $(d-1)$-faces are \emph{attaching} faces; let us denote these faces
by $\alpha_1$, \dots, $\alpha_t$. 
\nomenclature{$\iota_{(*)}$}{an initial face\nomrefpage}
\nomenclature{$\lambda_*$}{a liberation face\nomrefpage}
\nomenclature{$\alpha_*$}{an attaching face\nomrefpage}

We define $\BB$ by 
$$\BB = \TWO^S
\cup \CCg(\iota; \alpha_1, \dots, \alpha_t).$$ See Figure~\ref{FigSG} for a
schematic drawing.

\begin{figure}
\begin{center}
\epsfbox{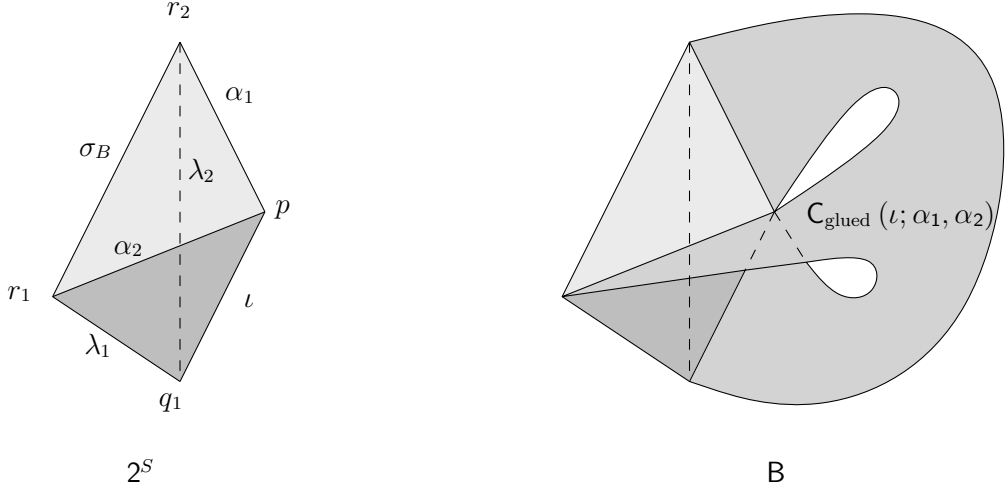}
\caption{A schematic drawing of the complexes $\TWO^S$ and $\BB$.}
\label{FigSG}
\end{center}
\end{figure}

\begin{proof}[Proof of Theorem~\ref{ThmOrder}(ii)]
We want to prove that $\BB$ is $d$-collapsible, but it contains a bad
$d$-collapsible face.

First, we observe that $\sigmab$ is a bad face. By
Lemma~\ref{LemCGlued}(ii) and
the inspection, the only $d$-collapsible
faces of $\BB$ are $\lambda_i$ and $\sigmab$ for $i \in [d]$. After collapsing
$\sigmab$ there is no $d$-collapsible face, implying that $\sigmab$ is
a bad face. 

In order to show $d$-collapsibility of $\BB$  we need a few other
definitions. 
The complex $\RR$ is defined by
$$
\RR = \setcond{\sigma \in \TWO^S}{\hbox{if } \set{q_1, \dots, q_{d-1}} \subseteq
\sigma \hbox{ then } \sigma \subseteq \iota}.
$$
We observe that $\RR \setminus \set{\iota}$ is $d$-collapsible and also that $\TWO^S
\collapseto \RR$ by collapsing all liberation faces (in any order). In fact,
the first observation is a special case of Lemma~\ref{LemSimGad}(ii) used for
the NP-reduction.

An auxiliary complex $\AA$ is defined in a similar way to $\BB$:

$$\AA = \RR \cup \CCg(\iota; \alpha_1, \dots, \alpha_t).$$

We show $d$-collapsibility of $\BB$ by the
following sequence of $d$-collapses:
$$\BB \collapseto \AA
\collapseto \RR \setminus \set{\iota} \collapseto \emptyset.$$

The fact that $\BB \collapseto \AA$ is quite obvious---it is sufficient to
$d$-collapse the liberation faces. More precisely, we use Lemma~\ref{LemSub} with
$\KK = \BB$, $\KK' = \TWO^S$, and $\LL' = \RR$. 
The fact that $\AA \collapseto \RR \setminus \set{\iota}$ follows from Lemma~\ref{LemCGlued}(i).
We already observed that $\RR \setminus \set{\iota} \collapseto \emptyset$ when defining $\RR$.
\end{proof}

\subsection{The complex $\CC$}
\nomenclature{$\CC$}{the connecting gadget\nomrefpage}
Our proof relies on
constructing $d$-dimensional $d$-collapsible complex $\CC$ such that its first
$d$-collapse is unique. We call this complex a \emph{connecting gadget}.
Precise properties of the connecting gadget are stated in Proposition~\ref{ProProMd}.

Before stating the proposition we define the notion of \emph{distant faces}. 
Suppose that $\KK$ is a simplicial complex and let $u$, $v$ be two of its vertices. By
$\dist(u,v)$ we mean their distance in graph-theoretical sense in the $1$-skeleton
of $\KK$. We say that two faces $\omega, \eta \in \KK$ are \emph{distant}
if $\dist(u, v) \geq 3$ for every $u \in \omega, v \in \eta$. 


\begin{proposition}
\label{ProProMd}
Let $d \geq 2$ and $t \geq 0$ be integers. There is a $d$-dimensional complex
$\CC = \CC
(\rho; \zeta_1, \dots, \zeta_t)$ with the following properties:

\begin{itemize}
\item[\emph{(i)}] It contains $(d-1)$-dimensional faces $\rho, \zeta_1, \dots,
\zeta_t$ such that each two of them are distant faces.
\item[\emph{(ii)}] Let $\CC' = \CC'(\rho; \zeta_1, \dots, \zeta_t)$ be the
subcomplex of $\CC$ given by $\CC' = \TWO^\rho \cup \TWO^{\zeta_1} \cup
\dots \cup \TWO^{\zeta_t}$. Then $\CC \collapseto (\CC' \setminus \set{\rho})$. In particular,
$\CC$ is $d$-collapsible since $(\CC' \setminus \set{\rho})$ is $d$-collapsible.

\item[\emph{(iii)}] The only $d$-collapsible face of $\CC$ is the face $\rho$.

\item[\emph{(iv)}] Suppose that $d$ is a constant. Then the number of faces of
$\CC$ is $O(t)$.
\end{itemize}
\end{proposition}

\subsection{The complex $\CC(\rho)$}
We start our construction assuming $t = 0$; i.e., we construct the connecting
gadget $\CC
= \CC(\rho)$. 

We remark that the construction of $\CC$ is in some respects similar to the
construction of generalized dunce hats. We refer
to~\cite{andersen-marjanovic-schori93} for more background.
 

\heading{\boldmath The geometric realization of  $\CC(\rho)$.}
\nomenclature{${\norm \KK \norm}$}{the geometric realization of a complex
$\KK$\nomrefpage}
First, we describe the geometric realization, $\| \CC \|$, of
$\CC$. 
Let $P$ be the $d$-dimensional \emph{crosspolytope}, the convex hull 
$$
\conv \set{\ee_1, - \ee_1, \dots, \ee_d, - \ee_d}
$$
of the vectors of the standard orthonormal basis and their negatives. It has
$2^d$ \emph{facets} 
$$
F_{\ss} = \conv \set{ s_1 \ee_1, \dots, s_d \ee_d},
$$
where $\ss = (s_i)_{i = 1}^d \in \set{-1, 1}^d$ ($\ss$ for \emph{sign}).
We want to glue all facets together except the facet $F_\uu$ where $\uu = (1,
\dots, 1)$ (see Figure~\ref{FigX23}). 

More precisely, let $\ss \in \set{-1,1}^d \setminus \set{\uu}$. 
Every $\xx \in F_{\ss}$ can be uniquely written as a convex combination $\xx =
\xx_{\aa, \ss} = a_1 s_1 \ee_1 +
\cdots + a_d s_d \ee_d$ where $\aa = (a_i)_{i=1}^d \in [0,
1]^d$ and $\sum_{i=1}^d a_i = 1$.
For every such fixed  $\aa$ we glue together the points in the
set $\setcond{\xx_{\aa, \ss}}{\ss \in \set{-1,1}^d\setminus\set{\uu}}$; by
$X$ we denote the resulting space. We will construct $\CC$ in such a
way that $X$ is a
geometric realization of $\CC$.

\begin{figure}
\begin{center}
\epsfbox{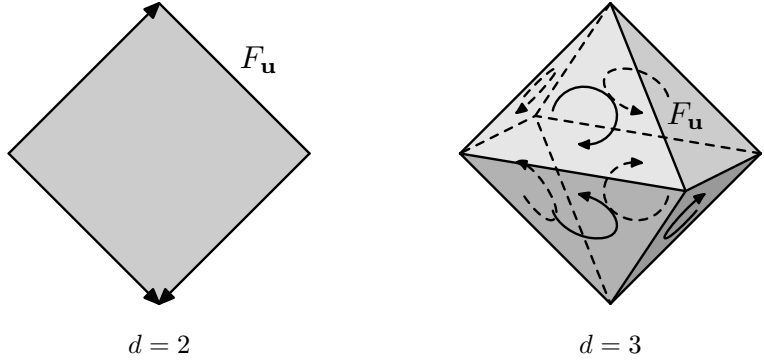}
\caption{The space $X$. The arrows denote, which facets are identified.}
\label{FigX23}
\end{center}
\end{figure}

\heading{Triangulations of the crosspolytope.}
We define two auxiliary triangulations of $P$---they are depicted in
Figure~\ref{FigTria}. The simplicial complex
$\JJ$ is the simplicial complex with vertex set $ 
\set{\mathbf 0, \mathbf e_1, - \mathbf e_1, \dots, \mathbf e_d, - \mathbf e_d}$. The
set of its faces is given by the maximal faces
$$
\set{\mathbf 0, s_1 \mathbf e_1, s_2 \mathbf e_2, \dots, s_d \mathbf
e_d} \hbox{ where } s_1, s_2, \dots, s_d \in \set{-1,1}.
$$
The complex $\JJ$ is a triangulation of $P$.

Let $\vartheta$ be the face $\set{\mathbf 0, \mathbf e_1, \dots, \mathbf
e_d}$.
The complex $\HH$ is constructed by
iterated \emph{stellar subdivisions} starting with $\JJ$ and subdividing faces of
$\JJ \setminus \TWO^\vartheta$ (first subdividing $d$-dimensional faces, then
$(d-1)$-dimensional, etc.).
Formally, $\HH$
is a complex with the vertex set $(\JJ \setminus \TWO^\vartheta) \cup
\vartheta$ and with faces of the form
$$
\set{\set{\sigma_1, \dots, \sigma_k} \cup \tau}
\hbox{ where }
 \sigma_1 \supsetneq \cdots \supsetneq \sigma_k \supsetneq \tau; \sigma_1,
\dots, \sigma_k \in \JJ \setminus \TWO^\vartheta; \tau \subseteq \vartheta; k \in \en_0. 
$$

\begin{figure}
\begin{center}
\epsfbox{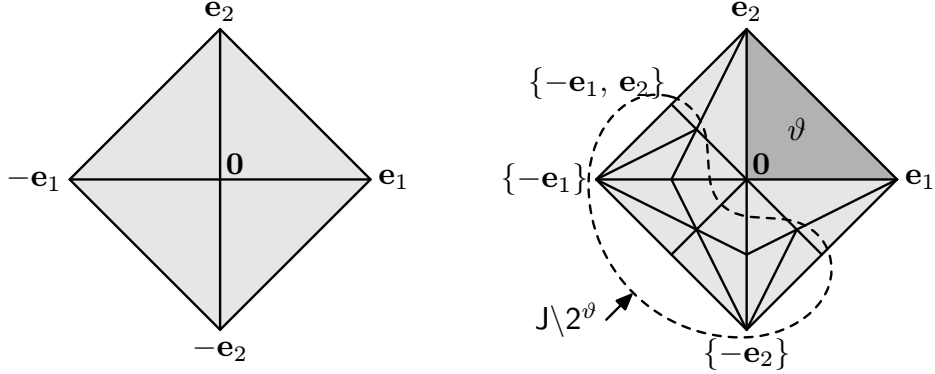}
\caption{The triangulations $\JJ$ (left) and $\HH$ (right) of $P$ with $d = 2$.}
\label{FigTria}
\end{center}
\end{figure}

\heading{\boldmath The construction of $\CC$.}
Informally, we obtain $\CC$ from $\HH$ by the same gluing as was used for 
constructing $X$ from $P$.

Formally, let $\approx$ be an equivalence relation on $(\JJ \setminus \TWO^\vartheta) \cup
\vartheta$ given by\\
\begin{tabular}{rcll}
$\mathbf e_i$ & $\approx$ & $\set{ -\mathbf e_i}$ & for $i \in [d]$, \cr
$\sigma_1$ & $\approx$ & $\sigma_2$ & for $\sigma_1, \sigma_2 \in \JJ
\setminus \TWO^\vartheta$, \cr
& & &
$\sigma_1 = \set{s_1 \ee_{k_1}, \dots, s_m
\ee_{k_m}}$,
$\sigma_2 = \set{s'_1 \ee_{k_1}, \dots, s'_m \ee_{k_m}}$ \cr
& & &
where $s_i, s'_i \in
\set{-1, 1}$ and $1 \leq k_1 < \dots < k_m \leq d$.
\cr 
\end{tabular}
 
For an equivalence relation $\equiv$ on a set $X$ we define $\equiv^+$ to be
an equivalence relation on $\mathcal Y \subset 2^X$ inherited from $\equiv$;
i.e., we have, for $Y_1, Y_2 \in \mathcal Y$, $Y_1 \equiv^+ Y_2$ if and only if 
there is a bijection $f\colon Y_1 \rightarrow Y_2$ such that $f(y) \equiv y$ for
every $y \in Y_1$.

We define $\CC = \HH /_{\approx^+}$. One can prove that $\CC$ is
indeed a simplicial complex and also that $\| \CC \|$ is homeomorphic to
$X$ (since the identification $\CC = \HH /_{\approx^+}$ was chosen to
follow the construction of $X$).

The faces of $\CC$ are the equivalence classes of $\approx^+$.
We use the notation $\langle \sigma \rangle$ for such an equivalence class
given by $\sigma \in \HH$. By $\rho$ we denote the
face $\langle \set{ \ee_1, \dots, \ee_d} \rangle$ of $\CC$.



\subsection{The complex $\CC (\rho; \zeta_1, \dots, \zeta_t)$}

Now we assume that $t \geq 1$ and we construct the complex $\CC (\rho; \zeta_1,
\dots, \zeta_t)$, which is a refinement of $\CC(\rho)$. The idea of the
construction is quite simple. We pick an interior simplex of $\CC(\rho)$; and
we subdivide it in such a way that we obtain distant $(d-1)$-dimensional
faces $\zeta_1, \dots, \zeta_t$ (and also distant from $\rho$). For completeness we show a particular way how
to get such a subdivision.

\heading{A suitable triangulation of a simplex.}
An example of the following construction is depicted in Figure~\ref{FigDdz}.
Let $\Delta$ be a $d$-dimensional (geometric) simplex with a set of vertices $V
= \set{\vv_1, \dots, \vv_{d+1}}$, let $\bb$ be its barycentre, and let $t$ be
an integer. Next, we define
$$ 
W = \setcond{\ww_{i,j}}{\ww_{i,j}
= \bb + \frac j{3t} (\vv_i - \bb), i \in [d+1], j \in [3t]}.
$$ Note that $V \subset  W$.
For $j \in [t]$, $\zeta_j$ is a
$(d-1)$-face $\set{\ww_{1,3j-2}, \ww_{2, 3j-2}, \dots, \ww_{d,
3j-2}}$. 

Now we define polyhedra $Q_1, \dots, Q_{3t}$. The polyhedron $Q_1$ is the
convex hull $\conv\set{\ww_{1,1} \dots, \ww_{d+1,1}}$.  For $j \in [3t] \setminus
\set{1}$ the polyhedron $Q_j$ is the union of the convex hulls
$$\bigcup\limits_{i \in [d+1]} \conv\setcond{\ww_{k,l} }{k \in [d+1] \setminus \set{i}, l \in \set{j-1, j} }.$$
The polyhedron $Q_1$ is a simplex. For $j > 1$, the polyhedra $Q_j$ are isomorphic
to the prisms $\partial \Delta^d \times [0,1]$, where
$\Delta^d$ is a $d$-simplex. Each such prism admits a (standard) 
triangulation such that
$\partial \Delta^d \times \set{0}$ and $\partial \Delta^d \times \set {1}$ are
not subdivided (see~\cite[Exercise~3, p.~12]{matousek03}).

\begin{figure}
\begin{center}
\epsfbox{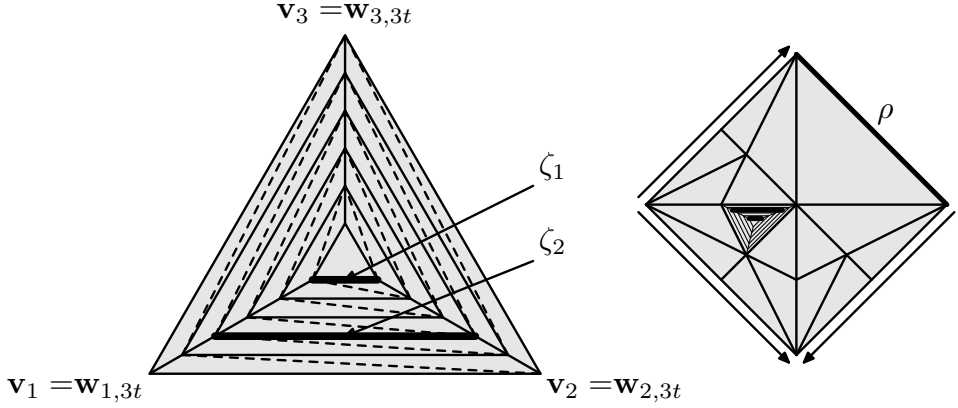}
\caption{The complex $\DD(\zeta_1, \zeta_2)$ (left) and $\CC(\rho; \zeta_1,
\zeta_2)$ (right), here $d = 2$.}
\label{FigDdz}
\end{center}
\end{figure}

Let $\DD(\zeta_1, \dots, \zeta_t)$ denote an abstract simplicial
complex on a vertex set $W$, which comes from a triangulation of $\Delta$ obtained by first
subdividing it into the polyhedra $Q_1, \dots, Q_{3t}$ and subsequently
triangulating these polyhedra as described above.



\heading{\boldmath The definition of $\CC(\rho; \zeta_1, \dots, \zeta_t)$.}
Let $\xi$ be a $d$-face of $\HH$ such that $\| \xi \| \subset \interior
\|\HH\|$. 
Although there are multiple such $d$-faces only some of them are
used as $\xi$. For example, in Figure~\ref{FigDdz}, only one out of four such
$d$-faces is chosen. 
Suppose that the set $V$ (from above) is the set of vertices of $\xi$. We define 
$$
\CC(\rho; \zeta_1, \dots, \zeta_t) = (\CC(\rho) \setminus \set{\langle \xi
\rangle }) \cup \DD(\zeta_1,
\dots, \zeta_t)
$$
while recalling that $\langle \xi \rangle$ denotes the equivalence class of
$\approx^+$ from the definition of $\CC$. 
See Figure~\ref{FigDdz}.

\begin{proof}[Proof of Proposition~\ref{ProProMd}]
The claims (i), (iii) and (iv) follow straightforwardly from the construction.
Regarding the claim (ii), informally, we first $d$-collapse the face $\rho$;
after that we $d$-collapse the ``interior'' of $\CC$ in order to collapse all
$d$-dimensional faces except the faces that should remain in $\CC' \setminus
\set{\rho}$.
Formally, we use Lemma~\ref{LemGra}.
\end{proof}

\heading{Gluing.}
Here we focus on gluing briefly discussed above Lemma~\ref{LemCGlued}.
As the name of connecting gadget suggests, we want to use it (in
Section~\ref{SecNP}) for connecting several other complexes (gadgets). In
particular, we want to have some notation for gluing this gadget. We introduce
this notation here.

\nomenclature{$\CCg$}{a complex $\CC$ glued to a given complex
$\KK$\nomrefpage}

Again we suppose that $\sigma$, $\gamma_1,
\dots, \gamma_t$ are already known $(d-1)$-dimensional faces of a given complex
$\LL$.
They are assumed to be distinct, but not necessarily disjoint. There is a
complex $\KK = \TWO^\sigma \cup \TWO^{\gamma_1} \cup \cdots
\cup \TWO^{\gamma_t}$. We take a
new copy of $\CC(\rho; \zeta_1, \dots, \zeta_t)$ and we perform identifications
$\rho = \sigma$, $\zeta_1 = \gamma_1, \dots, \zeta_t = \gamma_t$. After these
identifications, the complex $\KK \cup \CC(\rho; \zeta_1, \dots, \zeta_t)$ is
denoted by $\CCg (\sigma; \gamma_1, \dots, \gamma_t)$. Note that $\CC$ (before
gluing)  and $\CCg$ are generally not isomorphic since the gluing procedure
can identify some faces of $\CC$. 


\begin{proof}[Proof of Lemma~\ref{LemCGlued}.]
The first claim follows from Lemma~\ref{LemCon}. The second claim follows from Proposition~\ref{ProProMd}(i)
and (iii). The last claim follows from Pro\-po\-si\-tion~\ref{ProProMd}(iv).
\end{proof}


\section{NP-completeness}
\label{SecNP}

Here we prove Theorem~\ref{ThmMain}(ii). Throughout this section we assume
that $d \geq 4$ is a fixed integer. We have that $d$-COLLAPSIBILITY is in NP
since if we are given a sequence of faces of dimension at most $d-1$ we can
check in a polynomial time whether this sequence determine a $d$-collapsing of a given
complex.

For NP-hardness, we reduce the problem 3-SAT to $d$-COLLAPSIBILITY. The problem
3-SAT is NP-complete according to Cook~\cite{cook71}. Given
a 3-CNF formula $\Phi$, we construct a complex $\FF$ that is $d$-collapsible if and
only if $\Phi$ is satisfiable. 

\subsection{Sketch of the reduction}
Let us recall the construction of the bad complex $\BB$. We have started with a
simplex $\TWO^S$ and we distinguished the initial face $\iota$ and the bad face
$\sigmab$. We were allowed to start the collapsing either with $\sigmab$ or
with liberation faces and then with $\iota$. As soon as one of the
options was chosen the second one was unavailable. The idea is that these two
options should represent an assignment of variables in the formula $\Phi$. 

A disadvantage is that we cannot continue after collapsing $\sigmab$. Thus we
rather need to distinguish two initial faces $\iota^+$ and $\iota^-$ each of
them having its own liberation faces. However, we need that these two
collections of liberation faces do not interfere. That is why we have to assume
$d \geq 4$. 

For every variable $x_j$ of the formula $\Phi$ we thus construct a certain
\emph{variable gadget} $\VV_j$ with two initial faces $\iota^+_j$ and
$\iota^-_j$. For a clause $C^i$ in the formula $\Phi$ there is a \emph{clause
gadget} $\GG^i$. Initially it is not possible to collapse clause gadgets.
Assume, e.g., that $C^i$ contains variables $x_j$ and $x_{j'}$ in positive
occurrence and $x_{j''}$ in negative occurrence. Then it is possible to collapse
$\GG^i$ as soon as $\iota^+_j$, $\iota^+_{j'}$, or $\iota^-_{j''}$ was
collapsed. (This is caused by attaching a suitable copy of the connection gadget
$\CC$ from the previous section.) Thus the idea is that the complex $\FF$ in
the reduction is collapsible if and only if all clause gadgets can be
simultaneously collapsed which happens if and only if $\Phi$ is satisfiable.

There are few more details to be supplied. Similarly as for the construction of
$\BB$ we have to attach a copy $\TT$ of the connecting gadget $\CC$ to the faces
which are neither initial nor liberation (i.e., to attaching faces). This step
is necessary for controlling which faces can be collapsed. This copy of
connecting gadget is called a \emph{tidy connection} and once it is activated
(at least on of its faces is collapsed) then it is consequently possible to
collapse the whole complex $\FF$. Finally, there are inserted certain gadgets
called \emph{merge gadgets}. Their purpose is to merge the information obtained
by clause gadgets: they can be collapsed after collapsing all clause gadgets
and then they activate the tidy connection. The precise definition of $\FF$
will be described in following subsections. At the moment it can be helpful for
the reader to skip few pages and look at Figure~\ref{FigFF} (although there is
a notation on the picture not introduced yet).

\subsection{Simplicial gadgets}
Now we start supplying the details. 
As sketched above we introduce several gadgets called \emph{simplicial
gadgets}. They consist of full simplices (on
varying number of vertices) with several distinguished $(d-1)$-faces. These
gadgets generalize the complex $\TWO^S$. Every
simplicial gadget contains one or more $(d-1)$-dimensional pairwise disjoint
\emph{initial} faces. Every initial face $\iota$ contains several (possibly only
one) distinguished $(d-2)$-faces called \emph{bases} of $\iota$. 
The \emph{liberation} faces of the gadget are such $(d-1)$-faces $\lambda$
that contain a base of some initial face $\iota$, but $\lambda \neq
\iota$. The remaining $(d-1)$-faces are
\emph{attaching} faces.

Now we define several concrete examples of simplicial gadgets.
\begin{figure}
\begin{center}
\epsfbox{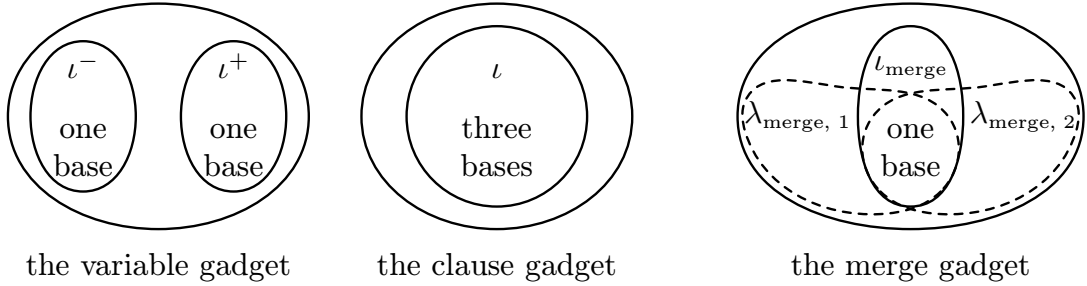}
\caption{A schematic pictures of simplicial gadgets; the liberation faces of the
merge gadget are distinguished.}
\end{center}
\end{figure}

\nomenclature{$\VV_{(*)}$}{a variable gadget\nomrefpage}

\heading{The variable gadget.}
The \emph{variable gadget} $\VV = \VV(\iota^+, \iota^-, \beta^+, \beta^-)$ is described by the
following table:

\begin{tabular}{ll}
vertices: & $p^+, q_1^+, \dots, q_{d-1}^+$, $p^-, q_1^-, \dots, q_{d-1}^-$; \cr
initial faces: & $\iota^+ = \set{p^+, q_1^+, \dots, q_{d-1}^+}$,  
 $\iota^- = \set{p^-, q_1^-, \dots, q_{d-1}^-}$; \cr
bases: & $\beta^+ = \set{q_1^+, \dots, q_{d-1}^+}$, 
  $\beta^- = \set{q_1^-, \dots, q_{d-1}^-}$. \cr
\end{tabular}\\[2mm]

\nomenclature{$\GG^{(*)}$}{a clause gadget\nomrefpage}

\heading{The clause gadget.}
The \emph{clause gadget} $\GG(\iota, \lambda_1, \lambda_2, \lambda_3)$ is given by:

\begin{tabular}{ll}
vertices: & $p_1, \dots, p_d, q$; \cr
initial face: & $\iota = \set{p_1, p_2, \dots, p_d}$;  \cr
bases: & $\beta_1 = \iota \setminus \set{p_1}$, 
  $\beta_2 = \iota \setminus \set{p_2}$. 
 $\beta_3 = \iota \setminus \set{p_3}$.
\cr
\end{tabular}\\[2mm]

Every base $\beta_j$ is contained in exactly one liberation face $\lambda_j = \beta_j
\cup \set{q}$.

\nomenclature{$\MM^{(*)}$}{a merge gadget\nomrefpage}
\heading{The merge gadget.}
The \emph{merge gadget} $\MM(\iotamerge, \lambdamergeone, \lambdamergetwo)$ is given by:

\begin{tabular}{ll}
vertices: & $p_1, \dots, p_d, q, r$; \cr
initial face: & $\iotamerge = \set{p_1, p_2, \dots, p_d}$;  \cr
base: & $\iotamerge \setminus \set{p_1}$.
\cr
\end{tabular}\\[2mm]
The merge gadget contains exactly two liberation faces, which we denote
$\lambdamergeone$
and $\lambdamergetwo$.

We close this subsection by proving a lemma about $d$-collapsings of simplicial
gadgets.
 
\begin{lemma}
\label{LemSimGad}
Suppose that $\SS$ is a simplicial gadget, $\iota$ is its initial face,
$\beta \subseteq \iota$ is a base face, and $\lambda_1, \dots, \lambda_t$ are
liberation faces containing $\beta$. Then $d$-collapsing of $\lambda_1, \dots,
\lambda_t$ (even in any order) yields a complex $\RR$ such that 
\begin{itemize}
\item[\emph{(i)}]
$\iota$ is a maximal face of $\RR$;
\item[\emph{(ii)}] $\RR \setminus \set{\iota}$ is $d$-collapsible;
\item[\emph{(iii)}] $\RR \setminus \set{\iota} \collapseto \TWO^{\iota'}$ where $\iota'$ is an
initial face different from $\iota$ (if exists).
\end{itemize} 
\end{lemma}
\begin{proof}
We prove each of the claims separately.
\begin{itemize}
\item[(i)] Let $V$ be the set of vertices of $\SS$ and let $\lambda_{t+1} =
\iota$. We (inductively) observe that $d$-collapsing of faces $\lambda_1, \dots,
\lambda_k$ for $k \leq t$ yields a complex in which $\lambda_{k+1}$ is contained
in a unique maximal face $(V \setminus (\lambda_1 \cup \dots \cup \lambda_k))
\cup \beta$. This implies that $\RR$ is well defined and also finishes the
first claim since
$$
(V \setminus (\lambda_1 \cup \dots \cup \lambda_t)) \cup \beta = \iota.
$$

We remark that the few details skipped here are exactly the same as in the proof of
Lemma~\ref{LemLessFaces}.
\item[(ii)] We observe that $\beta$ is a maximal $(d-2)$-face of $\RR \setminus
\set{\iota}$ and $\SS_{\beta} = \RR \setminus \set{\iota, \beta}$, hence $\RR \setminus \set{\iota} \rightarrow \SS_{\beta}$. (We
recall that $\KK_{\sigma}$ denotes the resulting complex of an elementary
$d$-collapse $\KK \rightarrow \KK_{\sigma} = \KK \setminus [\sigma,
\tau(\sigma)]$.)
Next, $\SS_\beta \collapseto \SS_{\emptyset} = \emptyset$ by
Lemma~\ref{LemLessFaces}.
\item[(iii)] Similarly as before we have  $
\RR \setminus \set{\iota} \rightarrow \SS_{\beta}$. Let $v$ be a vertex of
$\beta$, we have $\SS_{\beta} \rightarrow \SS_{\set{v}}$ by
Lemma~\ref{LemLessFaces}. The complex $\SS_{\set{v}}$ is a full simplex ($\SS$
with removed $v$), this complex even 1-collapse to $\TWO^{\iota'}$ by
collapsing vertices of $V \setminus (\iota' \cup \set{v})$ (in any order).
\end{itemize}
\end{proof}

\subsection{The reduction}
Let the given 3-CNF formula be $\Phi = C^1 \wedge C^2 \wedge \cdots \wedge
C^n$,
where each $C^i$ is a clause with exactly three literals (we assume without
loss of generality that
every clause contains three different variables). Suppose that $x_1,
\dots, x_m$ are variables appearing in the formula. For every such variable
$x_j$ we take a fresh copy of the variable gadget and we denote it by $\VV_j =
\VV_j(\iota_j^+, \iota_j^-, \beta_j^+, \beta_j^-)$. For every clause $C^i$ containing variables
$x_{j_1}$, $x_{j_2}$ and $x_{j_3}$ (in a positive or negative occurrence) we take
a new copy of the clause gadget and we denote it by $\GG^i = \GG^i(\iota^i,
\lambda^i_{j_1}, \lambda^i_{j_2}, \lambda^i_{j_3})$. Moreover, for $C^i$ with $i \geq
2$, we also take a new copy of the merge gadget and we denote it $\MM^i =
\MM^i(\iotamerge^i,\lambdamergeone^i, \lambdamergetwo^i)$.

Now we connect these simplicial gadgets by glued copies of the connecting
gadget called \emph{connections}. 

Suppose that a
variable $x_j$ occurs positively in the clauses $C^{i_1}, \dots, C^{i_k}$.
We construct the \emph{positive occurrence connections} by setting 
\nomenclature{$\OO^{+/-}_*$}{a positive/negative occurrence
connection\nomrefpage}

$$
\OO_j^+ = \CCg(\iota_j^+; \lambda^{i_1}_j, \dots, \lambda^{i_k}_j).
$$
The \emph{negative occurrence connections} $\OO_j^-$ are constructed similarly (we use
$\iota_j^-$ instead of $\iota_j^+$; and we use clauses in which is $x_j$ in
negative occurrence).

\nomenclature{$\II^*_*$}{a merge connection\nomrefpage}

The \emph{merge connections} are defined by\\
\begin{tabular}{llll}
$\II^1_1$ & $=$ & $\CCg(\iota^1; \lambdamergetwo^2)$; & \cr
$\II^i_1$ & $=$ & $\CCg(\iota^i; \lambdamergeone^i)$ & where $i \in \set{2,\dots,
n}$; \cr
$\II^i_2$ & $=$ & $\CCg(\iotamerge^i; \lambdamergetwo^{i+1})$ & where $i \in \set{2,\dots,
n-1}$. \cr
\end{tabular} \\
For convenient notation we denote $\II^1_1$ also by $\II^1_2$.

\nomenclature{$\TT$}{the tidy connection\nomrefpage}
Finally, the \emph{tidy connection} is defined by 
$$
\TT = \CCg(\iotamerge^n; \alpha_1, \dots, \alpha_t)
$$
where $\alpha_1, \dots, \alpha_t$ are attaching faces of all simplicial gadgets
in the reduction, namely the variable gadgets $\VV_j$ for $j \in [m]$, the clause
gadgets $\GG^i$ for $i \in [n]$, and the merge gadgets $\MM^i$ for $i \in
\set{2, \dots, n}$.

The complex $\FF$ in the reduction is defined by
$$
\FF = \bigcup\limits_{j =1}^m \VV_j \cup \bigcup\limits_{i = 1}^n \GG^i
\cup \bigcup\limits_{i = 2}^n \MM^i \cup \bigcup\limits_{j=1}^m (\OO^+_j \cup
\OO^-_j) \cup  \bigcup\limits_{i=1}^n \II^i_1
\cup \bigcup\limits_{i=2}^{n-1} \II^i_2 \cup \TT.
$$

See Figure~\ref{FigFF} for an example.

We observe that the number of faces of $\FF$ is polynomial in the number of
clauses in the formula (regarding $d$ as a constant). Indeed, we see that
the number of gadgets (simplicial gadgets and connections) is even linear in the number of variables. Each simplicial gadget has a constant size. Each connection has at most linear size due to Lemma~\ref{LemCGlued}(iii). 

\begin{figure}
\begin{center}
\epsfbox{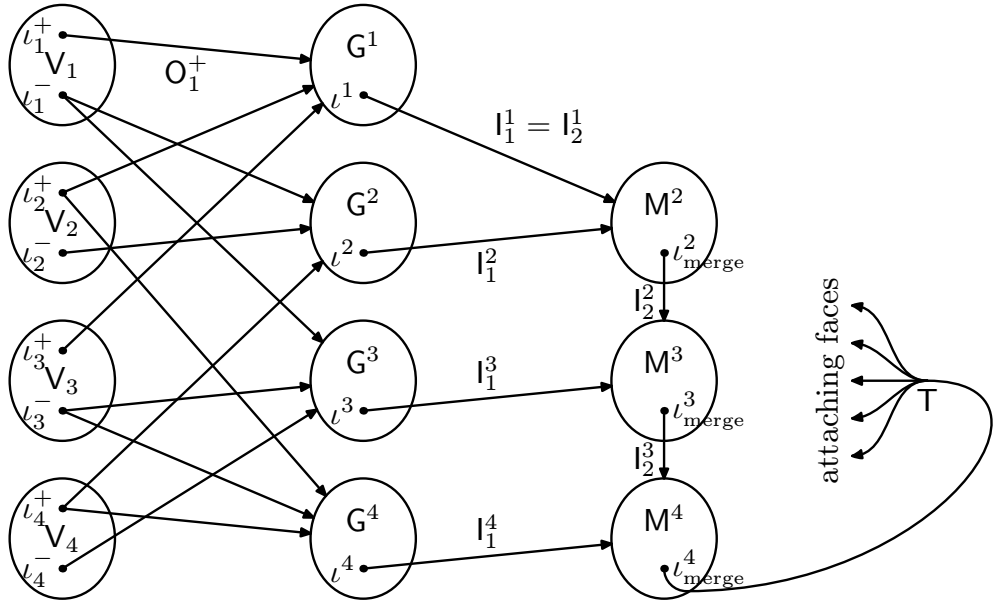}
\caption{A schematic example of $\FF$ for the
formula $\Phi = (x_1 \vee x_2 \vee x_3)\wedge(\neg x_1 \vee \neg x_2 \vee
x_4)\wedge(\neg x_1 \vee \neg x_3 \vee \neg x_4)\wedge(x_2 \vee \neg x_3 \vee
x_4)$. Initial faces are drawn as points. (Multi)arrows denote
connections. Each (multi)arrow points from the unique $d$-collapsible face of
the connection to simplicial gadgets that are attached to the connection by
some of its liberation faces.}
\label{FigFF}
\end{center}
\end{figure}

\begin{figure}
\begin{center}
\epsfbox{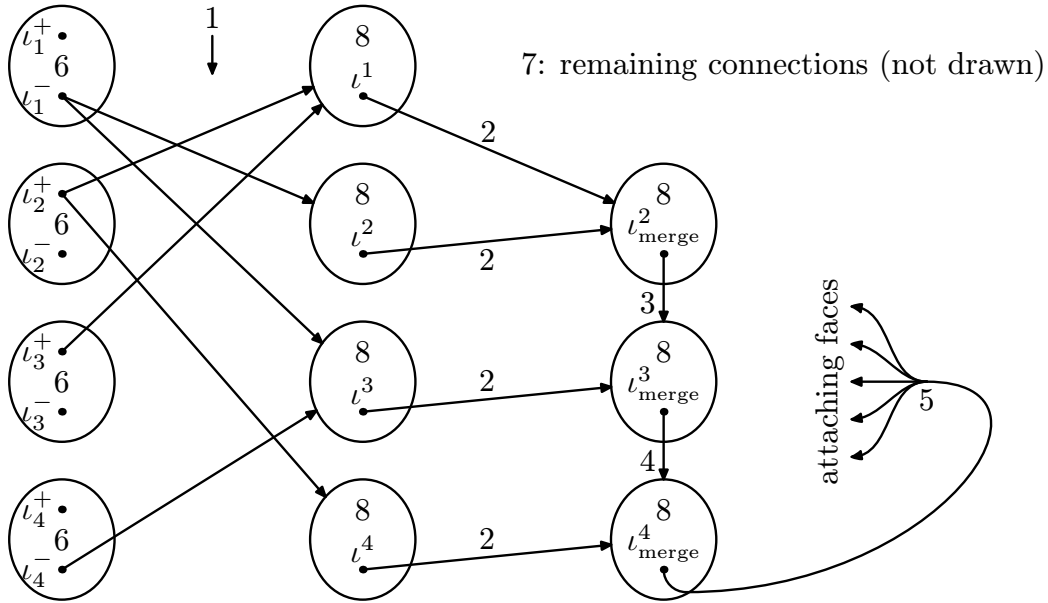}
\caption{
$d$-collapsing of $\FF$ for the $\Phi$ from Figure~\ref{FigFF} assigned 
(FALSE, TRUE, TRUE, FALSE). The numbers denote the order in which the parts of $\FF$
vanish.
}
\label{FigColSat}
\end{center}
\end{figure}

\heading{Collapsibility for satisfiable formulae.}
We suppose that the formula is satisfiable and we describe a collapsing of
$\FF$; see Figure~\ref{FigColSat}.

We assign each variable TRUE or FALSE so that the formula is satisfied.
For every variable gadget $\VV_j$ we proceed as follows. First, suppose that $x_j$
is assigned TRUE. We $d$-collapse\footnote{Note that after $d$-collapsing a
liberation face containing $\beta^+_j$ the liberation faces containing
$\beta^-_j$ are no more $d$-collapsible (and vice versa). This will be a key
property for showing that unsatisfiable formulae yield to non-collapsible
complexes.} the liberation faces containing $\beta^+_j$
(see Lemma~\ref{LemSimGad}(i)), after that $\iota^+_j$ is $d$-collapsible, and we $d$-collapse
$\OO^+_j$ (following Lemma~\ref{LemCGlued}(i) in the same way as in the proof of
Theorem~\ref{ThmOrder}(ii)). Similarly, we $d$-collapse
$\OO^-_j$ if $x_j$ is assigned FALSE. 

We use several times Lemma~\ref{LemSimGad}(i) and
Lemma~\ref{LemCGlued}(i) in the following paragraphs. The use is very similar
is in the previous one, thus we do not mention these lemmas again.

After $d$-collapsings described above, we have that every clause gadget $\GG^i$ contains at
least one liberation face that is $d$-collapsible since we have chosen such an
assignment that the formula is satisfied. We $d$-collapse this liberation face 
and after that the face $\iota^i$ is $d$-collapsible. We continue with
$d$-collapsing the merge gadgets $\II^i_1$ for $i \in [n]$.

The next step is that we gradually $d$-collapse the merge gadgets $\II^i_2$ for
$i \in \set{2, \dots, n-1}$. For this, we have that both liberation faces of
$\II^2_2$ are $d$-collapsible, we $d$-collapse them and we have that
$\iotamerge^2$ is $d$-collapsible. We $d$-collapse $\II^2_2$ and now we
continue with the same procedure with $\II^3_2$, the $\II^4_2$, etc. 

Finally, we $d$-collapse the tidy gadget. The $d$-collapsing of tidy gadget
makes all the attaching faces of simplicial gadgets $d$-collapsible. After this
``tidying up'' we can $d$-collapse all variable gadgets (using
Lemma~\ref{LemSimGad}(iii)), then all remaining connections, and at the end all
remaining simplicial gadgets due to Lemma~\ref{LemSimGad}(ii).

\heading{Non-collapsibility for unsatisfiable formulae.}
Now we suppose that $\Phi$ is unsatisfiable and we prove that $\FF$ is not
$d$-collapsible.

For contradiction, we suppose that $\FF$ is $d$-collapsible.
Let
$$
\FF = \FF_1 \rightarrow \FF_2 \rightarrow \cdots \rightarrow \emptyset
$$
be a $d$-collapsing of $\FF$. We call it \emph{our}
$d$-collapsing.
For a technical reason, according to
Lemma~\ref{LemBigFacesFirst}, we can assume that first $(d-1)$-dimensional faces
are collapsed and after that faces of less dimensions are removed.

Let us fix a subcomplex $\FF_\ell$ in our $d$-collapsing. Let $\NN$ be a connection (one of that forming
$\FF$) and let $\NN_\ell = \FF_\ell \cap \NN$. We say that $\NN$ is \emph{activated} in
$\FF_\ell$ if $\NN_\ell$ is a proper subcomplex of $\NN$.

The connection $\NN$ is defined as $\CCg (\sigma; \gamma_1, \dots, \gamma_s)$ for
some $(d-1)$-faces $\sigma, \gamma_1, \dots, \gamma_s$ of simplicial gadgets in
$\FF$. We remark that Lemma~\ref{LemCGlued}(ii) implies that if $\NN$ is
activated in $\FF_\ell$ then $\sigma \not \in \FF_\ell$. 


We also prove the following lemma about activated connections.

\begin{lemma}
\label{LemConActivated}
Let $\FF_\ell$ be a complex from our $d$-collapsing such that $\TT$
is not activated in $\FF_\ell$. Then we have:

\begin{itemize}
\item[\emph{(i)}] Let $j \in [m]$. If the positive occurrence connection $\OO^+_j$ is activated
in $\FF_\ell$, then the negative occurrence connection $\OO^-_j$ is not activated in
$\FF_\ell$ (and vice versa).

\item[\emph{(ii)}] Let $i \in [n]$. If the merge connection $\II^i_1$ 
is activated in $\FF_\ell$, then at least one of the three occurrence
connections attached to $\GG^i$ is activated in $\FF_\ell$.

\item[\emph{(iii)}] Let $i \in \set{2, \dots, n-1}$. If the merge connection
$\II^i_2$ 
is activated in $\FF_\ell$, then the merge connections $\II^i_1$ and $\II^{i-1}_2$
are activated in $\FF_\ell$.

\end{itemize}
\end{lemma}

\begin{proof}
Let us consider first $\ell - 1$ $d$-collapses of our $d$-collapsing
$$
\FF = \FF_1 \rightarrow \FF_2 \rightarrow \cdots \rightarrow \FF_\ell,
$$
where $\FF_{k+1} = \FF_k \setminus
[\sigma_k, \tau_k]$ for $k \in [\ell - 1]$. 
According to assumption on our $d$-collapsing, we have that $\sigma_1, \dots,
\sigma_{\ell - 1}$ are $(d-1)$-dimensional (since $\TT$ is not activated in
$\FF_\ell$ yet).

Now we prove each of the claims separately.
\begin{itemize}
\item[(i)]
For a contradiction we suppose that both $\OO^+_j$ and $\OO^-_j$ are activated in
$\FF_\ell$.
 
We consider the variable gadget $\VV_j$. 
We say that an index $k \in [\ell -
1]$ is \emph{relevant} if $\sigma_k \in \VV_j$. We observe that if $k$ is a
relevant index then $\sigma_k$ is a liberation face or an initial face of
$\VV_j$, because attaching faces are contained in $\TT$.

By \emph{positive} face we mean either the initial face $\iota^+_j$ or a
liberation face containing $\beta^+_j$. A \emph{negative} face is defined
similarly. Let $k^+$ (respectively $k^-$) be the smallest relevant index such that $\sigma_{k^+}$ is
a positive face (respectively negative face). These indexes have to exist since both
$\OO^+_j$ and $\OO^-_j$ are activated in $\FF_\ell$. Without loss of generality
$k^+ < k^-$. 

We show that $\sigma_{k^-}$ is not a $d$-collapsible face of $\FF_{k^- -1}$, thus we get a
contradiction. Indeed, let $S = \sigma_{k^+} \setminus \sigma_{k^-}$. We have $|S|
\geq 2$ since $d \geq 4$ (here we crucially use this assumption). Let $s \in
S$. Then we have $\sigma_{k^-} \cup \set{s} \in \FF_{k^- -1}$, because
$\sigma_{k^-} \cup \set{s}$ does not contain a positive subface (it does not contain $\beta^+_j$ since $|\sigma_{k^-} \cap \beta^+_j| \leq 1$,
but $|\beta^+_j| \geq 3$). On the other
hand $\sigma_{k^-} \cup S \not \in \KK_{k^- -1}$ since it contains
$\sigma_{k^+}$. I.e., $\sigma_{k^-}$ is not in a unique maximal face.

\item[(ii)]
We again define a relevant index; this time $k \in [\ell - 1]$ is \emph{relevant}
if $\sigma_k \in \GG^i$. We consider the smallest relevant index $k'$. Again we
have that $\sigma_{k'}$ is either the initial face $\iota^i$ or a liberation
face of $\GG^i$. In fact, $\sigma_{k'}$ cannot be $\iota^i$: by
Lemma~\ref{LemCGlued}(ii) we would have that $\II^i_1 \subseteq \FF_{k'-1}$ and
also $\GG^i \subseteq \FF_{k'-1}$ from minimality of $k'$, which would
contradict that $\sigma_{k'}$ is a collapsible face of $\FF_{k'-1}$. Thus
$\sigma_{k'}$ is a liberation face of $\GG^i$. This implies, again by
Lemma~\ref{LemCGlued}(ii), that at least one of the occurrence gadgets attached
to liberation faces is activated even in $\FF_{k'-1}$.

\item[(iii)] By a similar discussion as in previous step we have that at least
one of the liberation faces $\lambdamergeone^i$ and $\lambdamergetwo^i$ of
$\MM^i$ have to be $d$-collapsed before $d$-collapsing $\iotamerge^i$. However,
we observe that $d$-collapsing only one of these faces is still insufficient
for possibility of $d$-collapsing $\iotamerge^i$. Hence both of the liberation
faces have to be $d$-collapsed, which implies that both the gadgets $\II^i_1$
and $\II^{i-1}_2$ are activated in $\FF_{\ell}$.
\end{itemize}
\end{proof}

We also prove an analogy of Lemma~\ref{LemConActivated} for the tidy
gadget. We have to modify the assumptions, that is why we use a separate lemma.
The proof is essentially same as the proof of Lemma~\ref{LemConActivated}(iii),
therefore we omit it.

\begin{lemma}
\label{LemTidyActivated}
Let $\ell$ be the largest index such that $\TT$ is not activated in
$\FF_{\ell}$, then the merge connections $\II^n_1$ and $\II^{n-1}_2$ are
activated in $\FF_{\ell}$.
\end{lemma}
\proofend

Now we can quickly finish the proof of non-collapsibility. Let $\ell$ be the
integer from Lemma~\ref{LemTidyActivated}. By this lemma and repeatedly using
Lemma~\ref{LemConActivated}(iii) we have that all merge connections are
activated in $\FF_{\ell}$. By Lemma~\ref{LemConActivated}(ii), for every clause
gadget $\GG^i$ there is an occurrence gadget attached to $\GG^i$, which is
activated in $\FF_\ell$. Finally, Lemma~\ref{LemConActivated}(i) implies that
for every variable $x_j$ at most one of the occurrence gadgets $\OO^+_j$,
$\OO^-_j$ is activated in $\FF_\ell$. Let us assign $x_j$ TRUE if it is $\OO^+_j$ and
FALSE otherwise. This is satisfying assignment since for every $\GG^i$ at least one occurrence gadget attached to it is
activated in $\FF_\ell$. This contradicts the fact that $\Phi$ is unsatisfiable.

\proofend

\section{Technical properties of $d$-collapsing}
\label{AppTechnical}

In this section, we prove several auxiliary lemmas on $d$-collapsibility used
throughout the paper. 

\subsection{$d$-collapsing faces of dimension strictly less than $d-1$}

\begin{lemma}
\label{LemLessFaces}
Let $\KK$ be a complex, $d$ an integer, and $\sigma$ a $d$-collapsible face (in
particular, $\dim \sigma \leq d -1$). 
Let $\sigma' \supseteq \sigma$ be a face of $\KK$ of dimension at most $d-1$. Then
$\sigma'$ is $d$-collapsible and $\KK_{\sigma'} \collapseto \KK_\sigma$.
\end{lemma}

\begin{proof}
We assume that $\sigma \neq \sigma'$ otherwise the proof is trivial.

First, we observe that $\tau(\sigma)$ is a unique maximal face containing
$\sigma'$. Indeed, $\sigma' \subseteq \tau(\sigma)$ since $\tau(\sigma)$ is the
unique maximal face containing $\sigma$, and also if $\eta \supseteq \sigma'$,
then $\eta \supseteq \sigma$, which implies $\eta \subseteq \tau(\sigma)$. Hence
we have that $\sigma'$ is $d$-collapsible.

Let $v_1$ be a vertex of $\sigma' \setminus \sigma$. It is sufficient to prove
that $\KK_{\sigma'} \collapseto \KK_{\sigma' \setminus \set{v_1}}$ and proceed by
induction. Thus, for simplicity of notation, we can assume that $\sigma' =
\sigma \cup \set{v_1}$.

Let $v_2, \dots, v_t$ be vertices of $\tau(\sigma) \setminus \sigma'$. By
$\eta_i$ we denote the face $\sigma \cup \set{v_i}$ for $i \in [t]$. (In
particular, $\sigma' = \eta_1$.) For $i \in [t+1]$ we define a complex $\KK_i$ 
by the formula
$$
\KK_i = \setcond{\eta \in \KK}{\eta \not \supseteq \eta_1, \dots, \eta \not
\supseteq \eta_{i-1}}
=
\setcond{\eta \in \KK}{\hbox{if $\eta \supseteq \sigma$ then $v_j \not \in
\eta$ for $j < i$}}.
$$
From these descriptions we have that $\eta_i$ is a $d$-collapsible face of
$\KK_i$ contained in a unique maximal face $\tau_i = \tau(\sigma) \setminus
\set{v_1, \dots, v_{i-1}}$. Moreover $(\KK_i)_{\eta_i} = \KK_{i+1}$. Thus, we
have a $d$-collapsing
$$
\KK = \KK_1 \rightarrow \KK_2 \rightarrow \dots \rightarrow \KK_{t+1}.
$$
See Figure~\ref{FigLowFaces} for an example.

To finish the proof it remains to observe that $\KK_2 = \KK_{\sigma'}$ and
$\KK_{t+1}$ is a disjoint union of $\KK_\sigma$ and $\set{\sigma}$, hence $\KK_{t+1} \rightarrow \KK_{\sigma}$.
\end{proof}

\begin{figure}
\begin{center}
\epsfbox{complexdcollf10.eps}
\caption{An example of 2-collapsing $\KK \rightarrow \KK_{\sigma'} \collapseto
\KK_{\sigma}$.}
\label{FigLowFaces}
\end{center}
\end{figure}

As a corollary, we obtain the following lemma.

\begin{lemma}
\label{LemBigFacesFirst}
Suppose that $\KK$ is a $d$-collapsible complex. Then there is a $d$-collapsing
of $\KK$ such that first only $(d-1)$-dimensional faces are collapsed and after
that faces of dimensions less then $(d-1)$ are removed.
\end{lemma}

\begin{proof}
Suppose that we are given a $d$-collapsing of $\KK$. Suppose that in some step we
$d$-collapse a face $\sigma$ that is not maximal and its dimension is less than
$d - 1$. Let us denote this step by $\KK' \rightarrow \KK'_{\sigma}$. Let
$\sigma' \supseteq \sigma$ be such a face of $\KK'$ that either $\dim \sigma' = d - 1$ or $\sigma'$ is a maximal face. Then we
replace this step by $d$-collapsing $\KK' \rightarrow \KK'_{\sigma'}
\collapseto \KK_\sigma$.

We repeat this procedure until every $d$-collapsed face is either of dimension
$d-1$ or maximal. We observe that this procedure can be repeated only finitely
many times since in every
replacement we increase the number of elementary $d$-collapses in the
$d$-collapsing, while this
number is bounded by the number of faces of $\KK$.

Finally, we observe that if we first remove a maximal face of dimension less
than $d-1$ and then we $d$-collapse a $(d-1)$-dimensional face, we can swap
these steps with the same result.
\end{proof}

\subsection{$d$-collapsing to a subcomplex}

Suppose that $\KK$ is a simplicial complex, $\KK'$ is a subcomplex of it, which
$d$-collapses to a subcomplex $\LL'$. If certain conditions are satisfied, then
we can perform $d$-collapsing $\KK' \collapseto \LL'$ in whole $\KK$; see
Figure~\ref{FigSub} for an illustration. The precise statement is given in the
following lemma.

\begin{lemma}[$d$-collapsing a subcomplex]
\label{LemSub}
Let $\KK$ be a simplicial complex, $\KK'$ a subcomplex of $\KK$, and $\LL'$ a
subcomplex of $\KK'$. Assume that if $\sigma \in \KK' \setminus \LL'$,
$\eta \in \KK$, and $\eta \supseteq \sigma$, then $\eta \in \KK' \setminus \LL'$. Moreover
assume that $\KK' \collapseto \LL'$. Then $\LL = (\KK \setminus \KK') \cup \LL'$
is a simplicial complex and $\KK \collapseto \LL$. 
\end{lemma}

\begin{figure}
\begin{center}
\epsfbox{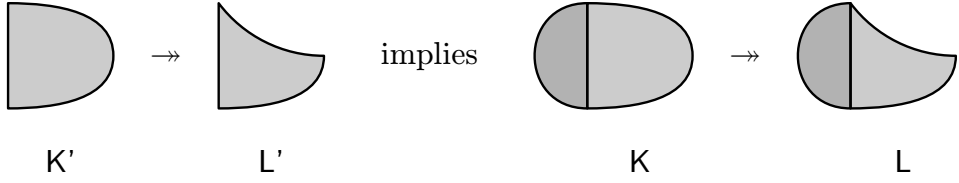}
\caption{Complexes $\KK$, $\KK'$, $\LL$ and $\LL'$ from the statement of
Lemma~\ref{LemSub}.}
\label{FigSub}
\end{center}
\end{figure}

\begin{proof}
It is straightforward to check that $\LL$ is a simplicial complex using the
equivalence\\
\centerline{$\eta \in \LL$ if and only if $\eta \in \KK$ and $\eta \notin
\KK' \setminus \LL'$.}\\[2mm]

In order to show $\KK \collapseto \LL$, it is sufficient to show the following
(and proceed by induction over elementary $d$-collapses):\\[2mm]
\emph{
Suppose that $\sigma'$ is a $d$-collapsible face of $\KK'$ such that
$\KK'_{\sigma'} \supseteq \LL'$. Then we have
\begin{itemize}
\item[$1.$] $\sigma'$ is a $d$-collapsible face of $\KK$.
\item[$2.$] If $\sigma \in \KK'_{\sigma'} \setminus \LL'$, $\eta \in
\KK_{\sigma'}$ and $\eta \supseteq \sigma$, then $\eta \in \KK'_{\sigma'}
\setminus \LL'$. 
\item[$3.$] $\LL = (\KK_{\sigma'} \setminus \KK'_{\sigma'})
\cup \LL'$.
\end{itemize}
}

We prove the claims separately:   

\begin{enumerate}
\item
We know that
$\sigma' \notin \LL'$ since $\KK'_{\sigma'} \supseteq \LL'$. Thus, $\sigma' \in
\KK' \setminus \LL'$. If $\eta' \in \KK$ and $\eta' \supseteq \sigma'$, then,
by the assumption of the lemma, $\eta' \in \KK' \setminus \LL' \subseteq \KK'$.
In~particular, the maximal faces in $\KK'$ containing $\sigma'$ coincide with
the maximal faces in $\KK$ containing $\sigma'$. It means that $\sigma'$ is a
$d$-collapsible face of $\KK$.

\item
We have $\KK'_{\sigma'} \setminus \LL' \subseteq  \KK' \setminus \LL'$ and
$\KK_{\sigma'} \subseteq \KK$. Thus the assumption of the lemma implies that
$\eta \in \KK' \setminus \LL'$. Next we have $\KK_{\sigma'} \cap \KK' =
\KK'_{\sigma'}$ since the maximal faces in $\KK'$ containing $\sigma'$ coincide
with the maximal faces in $\KK$ containing $\sigma'$. We conclude that $\eta
\in \KK'_{\sigma'} \setminus \LL'$.

\item
One can check that $\KK \setminus \KK' = \KK_{\sigma'} \setminus
\KK'_{\sigma'}$.
\end{enumerate}
\end{proof}

Suppose that $\Ff$ is a set system. For an integer $k$ we define the graph
$G_k(\Ff) = (V(G_k),E(G_k))$ as follows:\\[2mm]
\begin{tabular}{lcl}
$V(G_k)$ &=& $\setcond{F \in \Ff}{|F| = k + 1 \hbox{ (i.e., $\dim F = k$ if
$F$ is regarded as a face)}}$;\cr
$E(G_k)$ &=& $\setcond{\set{F, F'}}{F, F' \in V(G_k), \hbox{$F \cap F' \in
\Ff$ and $|F \cap F'| = k$}}$.\cr
\end{tabular} \\[2mm]

\begin{lemma}[$d$-collapsing a $d$-dimensional complex]
\label{LemGra}
Suppose that $\KK$ is a $d$-dimensional complex, $\LL$ is its subcomplex and
the following conditions are satisfied:
\begin{itemize}
\item $\KK \setminus \LL$ contains a $d$-collapsible face $\sigma$ such that
$\tau(\sigma) \in \KK \setminus \LL$;
\item $G_d(\KK \setminus \LL)$ is connected;
\item for every $(d-1)$-face $\eta \in \KK \setminus \LL$ there are at most two
$d$-faces in $\KK \setminus \LL$ containing $\eta$.  
\end{itemize}
Then $\KK \collapseto \LL$.
\end{lemma}

\begin{figure}
\begin{center}
\epsfbox{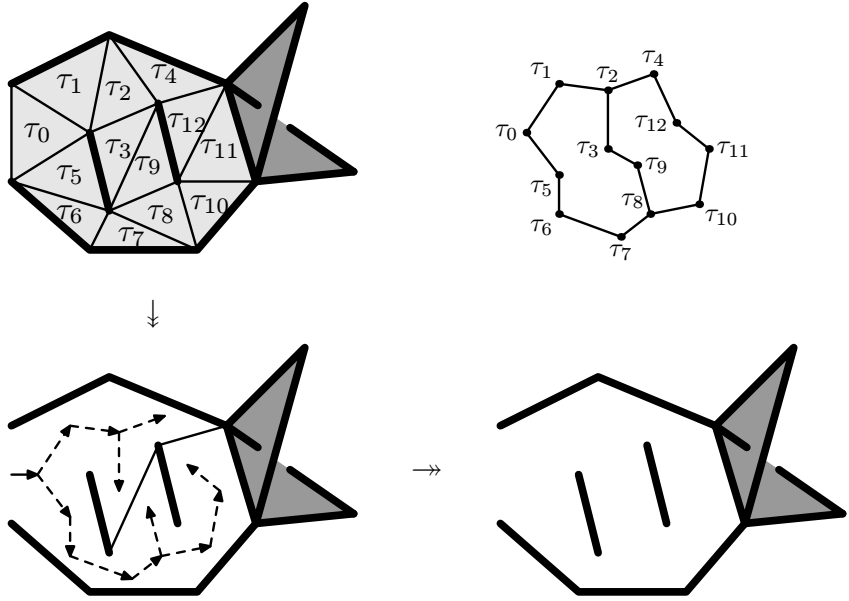}
\caption{In top right picture there are complexes $\KK$ and $\LL$ from
Lemma~\ref{LemGra}; $\LL$ is thick and dark. In top left picture there is the
graph $G_{2}(\KK \setminus \LL)$. Collapsing $\KK \collapseto \LL$ is in bottom
pictures.}
\label{FigConGraph}
\end{center}
\end{figure}

\begin{proof}
See Figure~\ref{FigConGraph} when following the proof.
Let $\tau_0 = \tau(\sigma)$, $\tau_1$, \dots, $\tau_j$ be an order of vertices
of $G_d(\KK \setminus \LL)$ such that for every $i \in [j]$ the vertex $\tau_i$
has a neighbor $\tau_{n(i)}$ with $n(i) < i$. Such an order exists by the
second condition. Let $\sigma_i = \tau_i \cap \tau_{n(i)}$. 

Consider the following sequence of elementary $d$-collapses\\[2mm]
\begin{tabular}{lcccl}
$\KK$ & $\rightarrow$ & $\KK_0$ & = & $\KK_{\sigma}$, \cr
$\KK_{i - 1}$ & $\rightarrow$ & $\KK_i$ & = & $(\KK_{i-1})_{\sigma_i}$ for $i
\in [j]$.
\cr
\end{tabular}\\[2mm]
This sequence is indeed a sequence of elementary $d$-collapses since
$\tau_{n(i)} \notin \KK_{i-1}$, thus $\tau_i$ is a unique maximal face
containing $\sigma_i$ in $\KK_{i-1}$ by the third condition. Moreover,
$\sigma_i \in \KK \setminus \LL$. Thus, $\KK_j$ is a supercomplex of $\LL$.
 
The set system $\KK_j \setminus \LL$ contains only faces of dimensions $d-1$ or
less. Hence $\KK_j \collapseto \LL$ by removing faces, which establishes the claim.
\end{proof}

\subsection{Gluing distant faces}
\label{SubGlu}

Let $k$ be an integer. Suppose that $\KK$ is a simplicial complex and let $\omega =
\set{u_1, \dots,
u_{k+1}}$, $\eta = \set{v_1, \dots, v_{k+1}}$ be two $k$-faces of $\KK$. 
By 
$$ \KK(\omega = \eta)$$
 we mean the resulting complex under the identification $u_1 = v_1, \dots,
u_{k+1} = v_{k+1}$ (note that this complex is not unique---it depends on the
order of vertices in $\omega$ and $\eta$; however, the order of vertices is not important for our purposes).

In a similar spirit, we define
$$ \KK(\omega_1 = \eta_1, \dots, \omega_t = \eta_t)$$
for $k$-faces $\omega_1, \dots, \omega_t, \eta_1, \dots, \eta_t$.


\begin{lemma}[Collapsing glued complex]
\label{LemColGlu}
Suppose that $\omega$ and $\eta$ are two distant faces in a simplicial complex
$\KK$. Let $\LL$ be a subcomplex of $\KK$ such that $\omega, \eta \in \LL$.
Suppose that $\KK$ $d$-collapses to $\LL$. Then $\KK(\omega = \eta)$
$d$-collapses to $\LL(\omega = \eta)$.

\end{lemma}


\begin{proof}
Let $\KK \rightarrow \KK_2 \rightarrow \KK_3 \rightarrow \cdots \rightarrow \LL$ be a
$d$-collapsing of $\KK$ to $\LL$. Our task is to show that 
$$\KK(\omega = \eta) \rightarrow \KK_2(\omega = \eta) \rightarrow
\KK_3(\omega = \eta) \rightarrow \cdots \rightarrow \LL(\omega =
\eta)$$ is a $d$-collapsing of $\KK(\omega \simeq \eta)$ to $\LL(\omega
\simeq \eta)$. 

It is sufficient to show $\KK(\omega = \eta) \rightarrow \KK_2(\omega =
\eta)$ and proceed by induction. 

For purposes of this proof, we distinguish faces before gluing $\omega = \eta$
by Greek letters, say $\sigma, \sigma'$, and after gluing by Greek letters in
brackets, say $[\sigma], [\sigma']$. E.g., we have $\omega \neq \eta$, but
$[\omega] = [\eta]$.

Suppose that $\KK_2 = \KK_{\sigma}$ for a
$d$-collapsible face $\sigma$. We want to show that $[\tau(\sigma)]$ is the
unique maximal face containing $[\sigma]$.
By the distance condition, we can without loss
of generality assume that $\sigma \cap \eta = \emptyset$ (otherwise we swap
$\omega$ and $\eta$). Suppose $[\sigma'
] \supseteq  [\sigma]$. Now we show that $\sigma' \supseteq \sigma$: if $\sigma
\cap \omega = \emptyset$ then $[\sigma] = \sigma$, and hence $\sigma' \subseteq
\sigma$ since the vertices of $\sigma$ are not glued to another vertices); if $\sigma \cap \omega \neq \emptyset$ then
$\sigma' \cap \eta = \emptyset$ due to the distance condition, which implies $\sigma'
\supseteq \sigma$.
Hence $\tau(\sigma) \supseteq \sigma'$, and
$[\tau(\sigma)] \supseteq [\sigma']$. Thus $[\tau(\sigma)]$ is the unique
maximal face containing $[\sigma]$.
\end{proof}

\begin{lemma}[Collapsing of the connecting gadget]
\label{LemCon}
Let $t$ be an integer. Let $\LLp$ be a complex with distinct $d$-dimensional
faces $\sigma$, $\gamma_1, \dots, \gamma_t$ such that $\sigma$ is a
maximal face of $\LLp$. Let $\CC = \CC(\rho, \zeta_1, \dots,
\zeta_t)$ and $\CC' = \CC'(\rho, \zeta_1, \dots, \zeta_t)$ be complexes defined in Section~\ref{SecBad}.

Then the complex 
$(\LLp \dot\cup \CC)(\sigma = \rho, \zeta_1 = \gamma_1, \dots, \zeta_t
= \gamma_t) 
$
$d$-collapses to the complex
$
(\LLp \dot\cup \CC')(\sigma = \rho, \zeta_1 = \varphi_1, \dots,
\zeta_t = \gamma_t) \setminus \set{\sigma}.
$

\end{lemma}

\begin{proof}

First, we observe that
$$
(\LLp \dot\cup \CC)(\sigma = \rho) \collapseto (\LLp\dot\cup \CC')(\sigma
= \rho) \setminus \set{\sigma}.
$$
This follows from Lemma~\ref{LemSub} by setting $\KK = (\LLp  \dot \cup
\CC)(\sigma = \rho)$, $\KK' = \CC$, $\LL' = \CC' \setminus \set{\sigma}$, and then
$\LL = (\LLp \dot\cup \CC')(\sigma = \rho) \setminus \set{\sigma}$. Assumptions of the lemma are satisfied by Proposition~\ref{ProProMd}(ii) and the inspection.

Now it is sufficient to iterate Lemma~\ref{LemColGlu}, assumptions are
satisfied by Proposition~\ref{ProProMd}(i).
\end{proof}


\section{The complexity of $d$-re\-pre\-sent\-abi\-lity}
\label{AppRepresentability}
In this section we prove that $d$-REPRESENTABILITY is NP-hard for $d \geq 2$.

\heading{Intersection graphs.}
Let $\Ff$ be a set system. The \emph{intersection graph} $I(\Ff)$ of $\Ff$ is defined
as the (simple) graph such that the set of its vertices is the set $\Ff$ 
and the set of its edges is the set $\setcond{\set{F, F'}}{F, F' \in
\Ff, F \neq F', F \cap F' \neq \emptyset}$. Alternatively, $I(\Ff)$ is
the 1-skeleton of the nerve of $\Ff$. 

A \emph{string graph} is a graph, which is
isomorphic to an intersection graph of finite collection of curves in the plane. By $\STR$ we denote
the set of all string graphs. By $\CON$ we denote the class of intersection
graphs of finite collections of convex sets in the plane, and by $\SEG$ we denote the class of
intersection graphs of finite collections of segments in the plane. Finally, by $\SEG(\leq2)$ we
denote the class of intersection graphs of finite collections of segments in
the plane such that no three segments share a common point.

Suppose that $G$ is a string graph. A system $\mathcal{C}$ of curves in the
plane such that $G$ is isomorphic $I(\mathcal C)$ is called an
\emph{$\STR$-representation} of $G$. Similar definitions apply to another
classes. We also establish a similar definition for simplicial complexes. Suppose that $\KK$ is a
$d$-representable simplicial complex. A system $\mathcal{C}$ of convex sets in
$\er^d$ such that $\KK$ is isomorphic to the nerve of $\mathcal{C}$ is called a
\emph{$d$-representation} of $\KK$.

We have $\STR \supseteq \CON
\supseteq \SEG$ (actually, it is known that the inclusions are strict).
Furthermore, suppose that we are given a graph $G
\in \SEG$. By Kratochv{\'\i}l and Matou\v{s}ek~\cite[Lemma
4.1]{kratochvil-matousek94}, there is a $\SEG$-representation of $G$ such that
no two parallel segments of this representation intersect. By a small
perturbation, we can even assume that no three segments of this representation
share a common point. Hence $\SEG = \SEG(\leq 2)$.

\heading{NP-hardness of 2-representability.}
Kratochv{\'\i}l and Matou\v{s}ek~\cite{kratochvil-matousek89} prove that for
the classes mentioned above (i.e., $\STR$, $\CON$ and $\SEG$) it is
NP-hard to recognize whether a given graph belongs to the given class. For this
they reduce planar 3-connected 3-satisfiability (P3C3SAT) to this problem
(see~\cite{kratochvil94} for the proof of NP-completeness of P3C3SAT and another
background). More precisely (see~\cite[the proof of Prop.
2]{kratochvil-matousek89}), given a formula $\Phi$ of P3C3SAT they construct
a graph $G(\Phi)$ such that $G(\Phi) \in \SEG$ if the formula is satisfiable,
but $G(\Phi) \not \in \STR$ if the formula is unsatisfiable. Moreover, 
we already know that this yields $G(\Phi) \in \SEG(\leq 2) $ for
satisfiable formulae.

Let us consider $G(\Phi)$ as a 1-dimensional simplicial complex. We will
derive that $G(\Phi)$ is 2-representable if and only if $\Phi$ is satisfiable.
 
If we are given a 2-representation of $G(\Phi)$ it is also a
$\CON$-representation of $G(\Phi)$ since $G(\Phi)$ is 1-dimensional. Hence
$G(\Phi)$ is not 2-representable for unsatisfiable formulae.

On the other hand, a $\SEG(\leq 2)$-representation of $G(\Phi)$ is also
a $2$-repre\-sentation of $G(\Phi)$. Thus $G(\Phi)$ is 2-representable for
satisfiable formulae.

In summary, we have that 2-REPRESENTABILITY is NP-hard.

\heading{\boldmath $d$-representability of suspension.} Let $\KK$ be a simplicial complex and let $a$ and $b$ be
two new vertices. By the \emph{suspension} of $\KK$ we mean the simplicial complex
$$
\susp \KK = \KK \cup \setcond{\set a \cup \sigma}{\sigma \in \KK} \cup
\setcond{\set b \cup \sigma}{\sigma \in \KK}.
$$

\begin{lemma}
\label{LemSusp}
Let $d$ be an integer. A simplicial complex $\KK$ is $(d-1)$-representable if
and only if $\susp \KK$ is $d$-representable.
\end{lemma}

\begin{proof}
First, we suppose that $\KK$ is $(d-1)$-representable and we show that $\susp
\KK$ is $d$-representable. Let $K_1, \dots, K_t \subseteq \er^{d-1}$ be convex
set from a $(d-1)$-repre\-sentation of $\KK$. Let $K(a)$ and $K(b)$ be
hyperplanes $\er^{d-1} \times \set{0}$ and $\er^{d-1} \times \set{1}$ in
$\er^d$. It is easy to see, that the nerve of the family 
$$
\set{K_1 \times [0,1], \dots, K_t \times [0,1], K(a), K(b)}
$$
of convex sets in $\er^d$ is isomorphic to $\susp \KK$.

For the reverse implication, we suppose that $\susp \KK$ is
$d$-representable and we show that $\KK$ is $(d-1)$-representable. Suppose that
$K(a), K(b), K_1 \dots, K_t$ is a $d$-representation of $\susp \KK$ ($K(a)$
corresponds to $a$ and $K(b)$ corresponds to $b$). We have that $\set{a,b} \not
\in \susp \KK$, thus there is a hyperplane $H \subseteq \er^d$ separating
$K(a)$ and $K(b)$ (we can assume that the sets in the representation are
compact). Then the nerve of the family
$$
\set{K_1 \cap H, \dots, K_t \cap H}
$$
of convex sets in $H \simeq \er^{d-1}$ is isomorphic to $\KK$.
\end{proof}

Since $2$-REPRESENTABILITY is NP-hard, we have the following corollary of
Lemma~\ref{LemSusp} (considering complexes that are obtained as $(d-2)$-tuple
suspensions):

\begin{theorem}
$d$-REPRESENTABILITY is NP-hard for $d \geq 2$.
\end{theorem}
\proofend

\section{Acknowledgement}
I would like to thank to Ji\v{r}\'{\i}  Matou\v{s}ek for discussing the
contents and reading the preliminary version of this paper. I would also like
to thank the anonymous referees for many improving comments.

\printnomenclature[2cm]

\bibliographystyle{alpha}
\bibliography{complexdcoll}

\begin{thebibliography}{AKMM02}

\bibitem[AK92]{alon-kleitman92}
{N}. Alon and {D}. Kleitman.
\newblock {P}iercing convex sets and the {H}adwiger {D}ebrunner $(p,
  q)$-problem.
\newblock {\em {A}dv. {M}ath.}, 96(1):103--112, 1992.

\bibitem[AKMM02]{alon-kalai-matousek-meshulam02}
N.~Alon, G.~Kalai, J.~Matou\v{s}ek, and R.~Meshulam.
\newblock {T}ransversal numbers for hypergraphs arising in geometry.
\newblock {\em {A}dv. in {A}ppl. {M}ath.}, 130:2509--2514, 2002.

\bibitem[Ame96]{amenta96}
{N}. Amenta.
\newblock {A} short proof of an interesting {H}elly-type theorem.
\newblock {\em {D}iscrete {C}omput. {G}eom}, 15:423--427, 1996.

\bibitem[AMS93]{andersen-marjanovic-schori93}
R.~N. Andersen, M.~M. Marjanovi{\'c}, and R.~M. Schori.
\newblock Symmetric products and higher-dimensional dunce hats.
\newblock {\em Topology Proc.}, 18:7--17, 1993.

\bibitem[{B}{\'{a}}r82]{barany82}
{I}. {B}{\'{a}}r\'{a}ny.
\newblock {A} generalization of {C}arath\'eodory's theorem.
\newblock {\em {D}iscrete {M}ath.}, 40:141--152, 1982.

\bibitem[Bj{\"o}95]{bjorner95}
A.~Bj{\"o}rner.
\newblock Topological methods.
\newblock In {\em Handbook of combinatorics, {V}ol.\ 1,\ 2}, pages 1819--1872.
  Elsevier, Amsterdam, 1995.

\bibitem[Bor48]{borsuk48}
Karol Borsuk.
\newblock On the imbedding of systems of compacta in simplicial complexes.
\newblock {\em Fund. Math.}, 35:217--234, 1948.

\bibitem[{C}oo71]{cook71}
{S}.~{A}. {C}ook.
\newblock The complexity of theorem proving procedures.
\newblock In {\em Proc. 3rd Ann. ACM Symp. on Theory of Computing}, pages
  151--158, 1971.

\bibitem[{H}at01]{hatcher01}
{A}. {H}atcher.
\newblock {\em {A}lgebraic Topology}.
\newblock {C}ambridge {U}niversity {P}ress, {C}ambridge, 2001.

\bibitem[{H}el23]{helly23}
{E}. {H}elly.
\newblock {\"{U}}ber mengen konvexer {K}{\"{o}}rper mit gemeinschaftlichen
  {P}unkten.
\newblock {\em {J}ahresber. {D}eustch. {M}ath.-{V}erein.}, 32:175--176, 1923.

\bibitem[Hel30]{helly30}
E.~Helly.
\newblock {\"U}ber {S}ysteme von abgeschlossenen {M}engen mit
  gemeinschaftlichen {P}unkten.
\newblock {\em {M}onaths. {M}ath. und {P}hysik}, 37:281--302, 1930.

\bibitem[KGT01]{kleitman-gyarfas-toth01}
D.~J. Kleitman, A.~Gy{\'a}rf{\'a}s, and G.~T{\'o}th.
\newblock Convex sets in the plane with three of every four meeting.
\newblock {\em Combinatorica}, 21(2):221--232, 2001.
\newblock Paul Erd{\H{o}}s and his mathematics (Budapest, 1999).

\bibitem[KL79]{katchalski-liu79}
{M}. {K}atchalski and {A}. {L}iu.
\newblock {A} problem of geometry in ${R}^n$.
\newblock {\em Proc. Amer. Math. Soc.}, 75:284--288, 1979.

\bibitem[KM89]{kratochvil-matousek89}
J.~Kratochv{\'\i}l and J.~Matou\v{s}ek.
\newblock {NP}-hardness results for intersection graphs.
\newblock {\em Coment. Math. Univ. Carolin.}, 30:761--773, 1989.

\bibitem[KM94]{kratochvil-matousek94}
J.~Kratochv{\'\i}l and J.~Matou\v{s}ek.
\newblock Intersection graphs of segments.
\newblock {\em {J}. {C}omb. {T}heory {S}er. B}, 62(2):289--315, 1994.

\bibitem[KM05]{kalai-meshulam05}
G.~Kalai and R.~Meshulam.
\newblock {A} topological colorful {H}elly theorem.
\newblock {\em {A}dv. {M}ath.}, 191(2):305--311, 2005.

\bibitem[KM08]{kalai-meshulam08}
G.~Kalai and R.~Meshulam.
\newblock {L}eray numbers of projections and a topological {H}elly type
  theorem.
\newblock {\em {J}. {T}opology}, 1(3):551--556, 2008.

\bibitem[Kra91]{kratochvil91}
J.~Kratochv{\'\i}l.
\newblock String graphs {II}. {R}ecognizing string graphs is {NP}-hard.
\newblock {\em {J}. {C}omb. {T}heory {S}er. B}, 52:67--78, 1991.

\bibitem[Kra94]{kratochvil94}
J.~Kratochv{\'\i}l.
\newblock A special planar satisfiability problem and a consequence of its
  {NP}-completeness.
\newblock {\em {D}iscrete {A}ppl. {M}ath.}, 50(3):297--302, 1994.

\bibitem[LB62]{lekkerkerker62}
C.~G. Lekkerkerker and J.~C. Boland.
\newblock Representation of a finite graph by a set of intervals on the real
  line.
\newblock {\em Fund. Math.}, 51:45--64, 1962.

\bibitem[{L}ov74]{lovasz74}
{L}. {L}ov\'asz.
\newblock {P}roblem 206.
\newblock {\em {M}atematikai {L}apok}, 25:181, 1974.

\bibitem[Mat03]{matousek03}
J.~Matou\v{s}ek.
\newblock {\em Using the {B}orsuk-{U}lam Theorem}.
\newblock {S}pringer, {B}erlin etc., 2003.

\bibitem[Mat09]{matousek09}
J.~Matou{\v{s}}ek.
\newblock Removing degeneracy in {LP}-type problems revisisted.
\newblock {\em Discrete Comput. Geom.}, 42(4):517--526, 2009.

\bibitem[MF08]{malgouyres-frances08}
R.~Malgouyres and A.~R. Franc{\'e}s.
\newblock Determining whether a simplicial 3-complex collapses to a 1-complex
  is {NP}-complete.
\newblock {\em DGCI}, pages 177--188, 2008.

\bibitem[MT09]{matousek-tancer09}
J.~Matou\v{s}ek and M.~Tancer.
\newblock Dimension gaps between representability and collapsibility.
\newblock {\em Discrete Comput. Geom.}, 42(4):631--639, 2009.

\bibitem[{M}un84]{munkres84}
{J}.~{R}. {M}unkres.
\newblock {\em {E}lements of Algebraic Topology}.
\newblock {A}ddison - {W}esley, 1984.

\bibitem[Tan10a]{tancer10prep}
M.~Tancer.
\newblock A counterexample to {W}egner's conjecture on good covers, 2010.
\newblock Manuscript in preparation.

\bibitem[{T}an10b]{tancer10}
{M}. {T}ancer.
\newblock {N}on-representability of finite projective planes by convex sets.
\newblock {\em {P}roc. {A}mer. {M}ath. {S}oc.}, 138(9):3285--3291, 2010.

\bibitem[Weg75]{wegner75}
G.~Wegner.
\newblock $d$-collapsing and nerves of families of convex sets.
\newblock {\em Arch. Math.}, 26:317--321, 1975.

\end{thebibliography}


\end{document}